\documentclass{amsart}

\title{Generic representations of abelian groups and extreme amenability}
\author{Julien Melleray and Todor Tsankov}

\address{Universit\'e de Lyon \\
CNRS UMR 5208 \\
Universit\'e Lyon 1 \\
Institut Camille Jordan \\
43 blvd. du 11 novembre 1918\\
F-69622 Villeurbanne Cedex \\
France}
\email{melleray@math.univ-lyon1.fr}

\address{Universit\'e Paris 7 \\ UFR de Math\'ematiques, case 7012 \\ Institut de Math\'ematiques de Jussieu \\ 75205 Paris \sc{cedex} 13 \\ France}
\email{todor@math.jussieu.fr}

\subjclass[2010]{Primary 54H11, 22A05; Secondary 54H05, 22F50, 54E52}
\keywords{Unitary representations, measure-preserving actions, Urysohn space, extremely amenable group, Polish group, Baire category}
\usepackage{amssymb}
\usepackage[initials,shortalphabetic]{amsrefs}
\usepackage[all]{xy}
\usepackage{mathpazo}

\usepackage{my-macros}
\numberwithin{equation}{section}

\DeclareMathOperator{\Res}{Res}

\begin{document}

\begin{abstract}
If $G$ is a Polish group and $\Gamma$ is a countable group, denote by $\Hom(\Gamma, G)$ the space of all homomorphisms $\Gamma \to G$. We study properties of the group $\cl{\pi(\Gamma)}$ for the generic $\pi \in \Hom(\Gamma, G)$, when $\Gamma$ is abelian and $G$ is one of the following three groups: the unitary group of an infinite-dimensional Hilbert space, the automorphism group of a standard probability space, and the isometry group of the Urysohn metric space. Under mild assumptions on $\Gamma$, we prove that in the first case, there is (up to isomorphism of topological groups) a unique generic $\cl{\pi(\Gamma)}$; in the other two, we show that the generic $\cl{\pi(\Gamma)}$ is extremely amenable. We also show that if $\Gamma$ is torsion-free, the centralizer of the generic $\pi$ is as small as possible, extending a result of Chacon--Schwartzbauer from ergodic theory.
\end{abstract}

\maketitle

\section{Introduction}
A \df{representation} of a countable group $\Gamma$ into a Polish group $G$ is a homomorphism $\Gamma \to G$. In the case when $G$ is the group of symmetries of a certain object $X$, a representation can be thought of as an action  $\Gamma \actson X$ that preserves the relevant structure; by varying $X$ one obtains settings of rather different flavor. For example, representation theory studies linear representations, while understanding homomorphisms $\Gamma \to \Aut(\mu)$ (or, equivalently, measure-preserving actions of $\Gamma$ on a standard probability space $(X, \mu)$) is the subject of ergodic theory.

In this paper, we take a global view and, rather than study specific representations, we are interested in the space of all representations $\Hom(\Gamma, G)$. $\Hom(\Gamma, G)$ is a closed subspace of the Polish space $G^\Gamma$ and is therefore also Polish. It is equipped with a natural action of $G$ by conjugation
\[
(g \cdot \pi)(\gamma) = g\pi(\gamma)g^{-1}, \quad \text{for } \pi \in \Hom(\Gamma, G), g \in G
\]
and if two representations are conjugate, one usually considers them as isomorphic. We are interested in \df{generic} properties of representations (that is, properties that hold for a comeager set of $\pi \in \Hom(\Gamma, G)$) that are moreover invariant under the $G$-action. We note that when $\Gamma$ is the free group $\F_n$, $\Hom(\Gamma, G)$ can be canonically identified with $G^n$ (and in particular, as a $G$-space, $\Hom(\Z, G)$ is isomorphic to $G$ equipped with the action by conjugation).

A very special situation is when the action $G \actson \Hom(\Gamma, G)$ has a comeager orbit; then, studying generic invariant properties of representations amounts to understanding the single generic point. Identifying when a comeager orbit exists has been the focus of active research during the last years, and looking for generic conjugacy classes is a well established topic in the study of automorphism groups of countable structures. Hodges--Hodkinson--Lascar--Shelah~\cite{Hodges1993a} introduced the notion of ample generics (a Polish group $G$ has \df{ample generics} if there is a generic element in $\Hom(\F_n, G)$ for every $n$) and gave the first examples, then Kechris--Rosendal~\cite{Kechris2007a} undertook a systematic study of this notion; more recently, Rosendal~\cite{Rosendal2010p2} proved that there is a generic point in $\Hom(\Gamma, \Iso(\Q\bU))$ for any finitely generated abelian $\Gamma$ ($\Q\bU$ is the rational Urysohn space).

In many situations, however, it turns out that orbits are meager. Important examples of this phenomenon come from ergodic theory (that is, when $G = \Aut(\mu)$), and identifying generic properties in that setting has a long history, especially in the case $\Gamma=\Z$: for example, Halmos proved that the generic transformation in $\Aut(\mu)$ is weakly mixing and has rank $1$. We refer the reader to Kechris's book ~\cite{Kechris2010} for a detailed study of the spaces $\Hom(\Gamma, \Aut(\mu))$ for various groups $\Gamma$.

In this paper, we study properties of the subgroup $\cl{\pi(\Gamma)} \leq G$ for a generic $\pi \in \Hom(\Gamma, G)$. Our first theorem is quite general and may be considered as a starting point for the rest of the results. Recall that a Polish group is called \df{extremely amenable} if every continuous action of the group on a compact Hausdorff space has a fixed point. We recommend Pestov's book~\cite{Pestov2006} as a general reference on extreme amenability.
\begin{theorem} \label{thi:gdelta}
Let $\Gamma$ be a countable group and $G$ a Polish group. Then the set
\[
\set{\pi : \cl{\pi(\Gamma)} \text{ is extremely amenable}}
\]
is $G_\delta$ in $\Hom(\Gamma, G)$.
\end{theorem}
Of course, it is possible that generically $\pi(\Gamma)$ is dense in $G$, for instance this often happens when $\Gamma$ is free. The easiest way to ensure that we are in a non-trivial situation is to assume that $\Gamma$ satisfies some global law, for example that it is abelian. For that reason (among others), in the rest of the paper, we specialize to abelian groups $\Gamma$ and three examples for $G$: the unitary group of a separable, infinite-dimensional Hilbert space $U(\mcH)$, the group of measure-preserving automorphisms of a standard probability space $\Aut(\mu)$, and the group of isometries of the universal Urysohn metric space $\Iso(\bU)$ (see \cite{TopAppl:Urysohn2008} for background on the Urysohn space). Conjugacy classes are meager in $\Hom(\Gamma, U(\mcH))$ for any countable infinite $\Gamma$ \cite{Kerr2010}; it is an open question whether this is true for the other two groups. The orbits in $\Hom(\Gamma,\Aut(\mu))$ are known to be meager for any infinite amenable $\Gamma$ \cite{Foreman2004}. For $\Iso(\bU)$, the situation is less clear, though we prove here, using a technique due to Rosendal, that conjugacy classes in $\Hom(\Gamma,\Iso(\bU))$ are meager for all torsion-free abelian $\Gamma$ (see Theorem~\ref{t:conj-meager}).

Of major importance in this paper is the group $L^0(\T)$: the Polish group of (equivalence classes of) measurable maps $([0, 1], \lambda) \to \T$, where $\lambda$ denotes the Lebesgue measure, $\T$ is the torus, the group operation is pointwise multiplication, and the topology is given by convergence in measure (equivalently, $L^0(\T)$ can be viewed as the unitary group of the abelian von Neumann algebra $L^\infty([0, 1])$). Two important facts about this group, proved by Glasner~\cite{Glasner1998a}, are that it is \df{monothetic} (contains a dense cyclic subgroup) and that it is extremely amenable.

Recall that an abelian group is called \df{bounded} if there is an upper bound for the orders of its elements and \df{unbounded} otherwise. Even though our results hold in a slightly greater generality, in order to avoid introducing technical definitions, we state them here only for unbounded groups.

In the case of the unitary group, using spectral theory, we are able to identify exactly the group $\cl{\pi(\Gamma)}$ for a generic $\pi$.
\begin{theorem} \label{thi:unitary}
Let $\Gamma$ be a countable, unbounded, abelian group. Then the set
\[
\set{\pi : \cl{\pi(\Gamma)} \cong L^0(\T)}
\]
is comeager in $\Hom(\Gamma, U(\mcH))$.
\end{theorem}

For our other two examples, we are unable to decide whether generically one obtains a single group.
\begin{question} \label{q:same-group}
For $G$ either $\Aut(\mu)$ or $\Iso(\bU)$, does there exist a Polish group $H$ such that for the generic $\pi \in \Hom(\Z, G)$, $\cl{\pi(\Z)} \cong H$? In particular, is this true for $H = L^0(\T)$?
\end{question}
This question was also raised (for $\Aut(\mu)$) by Solecki and by Pestov. Solecki~\cite{Solecki2012} recently found some evidence that the answer might be positive for $\Aut(\mu)$: he proved that the closed subgroup generated by a generic $\pi \in \Hom(\Z, \Aut(\mu))$ is a continuous homomorphic image of a closed subspace of $L^0(\R)$ and contains an increasing union of finite dimensional tori whose union is dense.

A related problem was posed by Glasner and Weiss~\cite{Glasner2005a}: is $\cl{\pi(\Z)}$ a L{\'e}vy group for the generic $\pi \in \Hom(\Z, \Aut(\mu))$? Recall that $L^0(\T)$ is L{\'e}vy and that every L{\'e}vy group is extremely amenable (see \cite{Pestov2006} for background on L{\'e}vy groups).

Even though the questions above remain open, using Theorem~\ref{thi:gdelta}, we can still show that extreme amenability is generic.
\begin{theorem} \label{thi:ext-amen-examples}
Let $\Gamma$ be a countable, unbounded, abelian group and $G$ be one of $\Aut(\mu)$ or $\Iso(\bU)$. Then the set
\[
\set{\pi : \cl{\pi(\Gamma)} \cong L^0(\T)}
\]
is dense in $\Hom(\Gamma, G)$ and therefore, the generic $\cl{\pi(\Gamma)}$ is extremely amenable.
\end{theorem}
In the special case $\Gamma = \Z$ and $G = \Iso(\bU)$, this answers a question of Glasner and Pestov (\cite{Pestov2007a}*{Question 8}). The methods used in the proof of the theorem also show that the \emph{generic monothetic group} (that is, the generic completion of $\Z$) is extremely amenable (see Theorem~\ref{t:extamisgdelta}).


A natural question that arises is how one can distinguish the generic groups $\cl{\pi(\Gamma)}$ for different abelian groups $\Gamma$. It turns out that, at least for finitely generated free abelian groups, this is impossible.
\begin{theorem}
Let $G$ be either $\Aut(\mu)$ or $\Iso(\bU)$, let $P$ be a property of abelian Polish groups such that the set $\set{\pi \in \Hom(\Z, G) : \cl{\pi(\Z)} \text { has } P}$ has the Baire property, and let $d$ be a positive integer. Then the following are equivalent:
\begin{enumerate} \romanenum
\item for the generic $\pi \in \Hom(\Z, G)$, $\cl{\pi(\Z)}$ has property $P$;
\item for the generic $\pi \in \Hom(\Z^d, G)$, $\cl{\pi(\Z^d)}$ has property $P$.
\end{enumerate}
\end{theorem}
In particular, the theorem implies that the generic $\cl{\pi(\Z^d)}$ is monothetic. It also implies that if the answer of Question~\ref{q:same-group} is positive for $\Z$, then it is positive for any $\Z^d$ and moreover, one obtains the same generic group. Finally, in the case $G = \Aut(\mu)$, a result of Ageev allows us to replace $\Z^d$ by any countable abelian $\Gamma$ containing an infinite cyclic subgroup (cf. Subsection~\ref{ss:AutM}).

Next we apply our techniques to studying stabilizers of representations. If $\pi \in \Hom(\Gamma, G)$, denote by $\mcC(\pi)$ the \df{centralizer} of $\pi$, i.e.,
\[
\mcC(\pi) = \set{g \in G : g \pi(\gamma) = \pi(\gamma) g \text{ for all } \gamma \in \Gamma}.
\]
Then $\mcC(\pi)$ is a closed subgroup of $G$ and $\cl{\pi(\Gamma)} \leq \mcC(\pi)$ if $\Gamma$ is abelian. Our next result shows that, for a generic $\pi$, we often have equality.
\begin{theorem} \label{thi:centralizers}
Let $G$ be either $\Aut(\mu)$ or $\Iso(\bU)$ and $\Gamma$ be a countable, non-trivial, torsion-free, abelian group. Then for a generic $\pi \in \Hom(\Gamma, G)$, $\mcC(\pi) = \cl{\pi(\Gamma)}$.
\end{theorem}
For $G=\Aut(\mu)$ and $\Gamma= \Z$, this result seems to have been stated explicitly for the first time in \cite{Akcoglu1970}, where it is attributed to Chacon-Schwartzbauer \cite{Chacon1969}\footnote{In an ealier version of this paper this result was incorrectly attributed to King, who proved a stronger statement, namely, that the conclusion of the theorem holds for rank-$1$ transformations. We are grateful to S\l awomir Solecki for making us aware of the papers \cite{Chacon1969} and \cite{Akcoglu1970}.}. Our approach is different in that it is based on maps between suitable Polish spaces which behave well with respect to Baire category notions, which we call \df{category-preserving maps}, and a generalization of the Kuratowski--Ulam theorem (see Appendix~\ref{a:category}). Such maps were first used in a similar setting by King~\cite{King2000}, were explicitly defined by Ageev~\cite{Ageev2000}, and also feature prominently in the papers Tikhonov~\cite{Tikhonov2003} and Stepin--Eremenko~\cite{Stepin2004}. Tikhonov~\cite{Tikhonov2003}, who proved that the generic action of $\Z^d$ can be embedded in a flow, was the first to consider actions of abelian groups other than $\Z$.

We conclude with an amusing application of the methods developed in this paper. All three target groups that we consider ($U(\mcH)$, $\Aut(\mu)$, $\Iso(\bU)$) are known to be L\'evy and therefore extremely amenable. Using Theorem~\ref{thi:gdelta}, we are able to give a ``uniform'' proof for their extreme amenability that does not pass through the fact that they are L\'evy groups (but it does use the extreme amenability of $L^0(K)$ for compact groups $K$, and those are L\'evy).
\begin{theorem} \label{thi:free-compact}
Let $G$ be either of $U(\mcH)$, $\Aut(\mu)$, $\Iso(\bU)$; denote by $U(n)$ the unitary group of dimension $n$ and by $\F_{\infty}$ the free group on $\aleph_0$ generators. Then the set
\[
\set{\pi \in \Hom(\F_\infty, G) : \exists n \ \cl{\pi(\F_\infty)} \cong L^0(U(n))}
\]
is dense in $\Hom(\F_\infty, G)$.
\end{theorem}
\begin{cor}[Gromov--Milman \cite{Gromov1983}; Giordano--Pestov \cite{Giordano2007}; Pestov \cite{Pestov2002}]
The three groups $U(\mcH)$, $\Aut(\mu)$, and $\Iso(\bU)$ are extremely amenable.
\end{cor}

\noindent {\bf Notation.} All groups are written multiplicatively and the neutral element is always denoted by $1$. When $G$ is a group and $A \sub G$, we denote by $\langle A \rangle$ the subgroup of $G$ generated by $A$. $(X, \mu)$ is an atomless standard probability space; we write $\Aut(X, \mu)$ or simply $\Aut(\mu)$ to denote the group of measure-preserving bijections of $X$ modulo a.e. equality (equipped with its usual Polish topology). $\mcH$ is a separable, infinite-dimensional, complex Hilbert space and $U(\mcH)$ denotes its unitary group endowed with the strong operator topology. If $G$ is a Polish group, $L^0(G)$ denotes the group of all (equivalence classes of) measurable maps from $(X,\mu)$ to $G$ with pointwise multiplication and the topology of convergence in measure (see \cite{Kechris2010}*{Section~19} for basic facts about those groups). Sometimes, when we need to explicitly mention the source measure space, we may write $L^0(X, \mu, G)$. $\Z$ denotes the infinite cyclic group; $\Z(n)$ is the finite cyclic group of order $n$; $\T$ is the multiplicative group of complex numbers of modulus $1$; $\F_n$ is the free group with $n$ generators ($n \le \aleph_0$; if $n=\aleph_0$ we use the notation $\F_{\infty}$); $U(m)$ is the unitary group of an $m$-dimensional Hilbert space.

\smallskip

\noindent {\bf Acknowledgements.} The research that led to this paper was initiated in the Fields Institute in Toronto during the Thematic Program on Asymptotic Geometric Analysis and we would like to thank the Institute as well as the organizers of the program for the extended hospitality and the excellent working conditions they provided. We are grateful to Vladimir Pestov for stimulating discussions on the topic of this paper, to Oleg Ageev for pointing out a mistake in a preliminary draft, and to S\l awomir Solecki for interesting bibliographical information; finally we would like to thank the anonymous referee for several corrections, pertinent historical remarks, and for drawing our attention to some relevant literature. Work on this project was partially supported by the ANR network AGORA, NT09-461407.

\section{Preliminaries on abelian groups}
In this section, we prove several preliminary lemmas about abelian groups that will be useful in the sequel.

By a classical theorem of Pr\"ufer (see \cite{Fuchs1970}*{Theorem~17.2}), every bounded countable abelian group is a direct sum of cyclic groups, moreover, the function
\[
m_\Gamma \colon \set{p^n : p \text{ prime}, n \in \N^+} \to \N \cup \set{\infty}
\]
indicating the number of summands isomorphic to $\Z(p^n)$ is a complete isomorphism invariant for bounded abelian groups; note that for such groups, $m_{\Gamma}$ has only finitely many non-zero values.

Consider the following property of a bounded abelian group:
\begin{equation} \tag{$\ast$} \label{eq:ast}
\forall p \forall n \ m_\Gamma(p^n) > 0 \implies \exists k \geq n \ m_\Gamma(p^k) = \infty.
\end{equation}
The reason we are interested in it is that a bounded abelian group without \eqref{eq:ast} cannot be densely embedded in an extremely amenable group. More precisely, we have the following.
\begin{prop} \label{p:discon}
Let $\Gamma$ be a bounded abelian group that does not have property \eqref{eq:ast}. Then any Hausdorff topological group $G$ in which $\Gamma$ can be embedded as a dense subgroup has a non-trivial open subgroup of finite index.
\end{prop}
\begin{proof}
Let $p$ and $n$ be such that $0 < m_\Gamma(p^n) < \infty$ and $m_\Gamma(p^m) = 0$ for all $m > n$. Let $k$ be the product of all orders of elements of $G$ not divisible by $p$ and consider the closed subgroup $H \leq G$ given by
\[
H = \set{g \in G : g^{kp^{n-1}} = 1}.
\]
Then $H \cap \Gamma$ has finite index in $\Gamma$ (the quotient $\Gamma/(H \cap \Gamma)$ being isomorphic to $\Z(p)^{m_\Gamma(p^n)}$), showing that $H$ has finite index in $G$ and is therefore open (if $\Gamma = \bigcup_{i=1}^n \gamma_i (H \cap \Gamma)$, then $\bigcup_{i=1}^n \gamma_i H$ is closed and contains $\Gamma$).
\end{proof}

If $\Gamma$ is bounded, denote by $H_\Gamma$ the subgroup of $\T$ generated by
\[
\set{\exp(2\pi i / p^n) : m_\Gamma(p^n) = \infty}
\]
and note that $H_\Gamma$ is finite, cyclic, and that if $\Gamma$ has property \eqref{eq:ast}, every character of $\Gamma$ takes values in $H_\Gamma$.

For any countable abelian group $\Gamma$, denote by $\hat \Gamma$ its dual, i.e., the group of all homomorphisms $\Gamma \to \T$, and recall that as $\Gamma$ is discrete, $\hat \Gamma$ is compact.
\begin{lemma} \label{l:densehatGamma}
Let $\Gamma$ be a countable abelian group and let $e \colon \Gamma \to \Gamma^n$ be the diagonal embedding.
\begin{enumerate} \romanenum
\item \label{i:dhG:ub} If $\Gamma$ is unbounded, then the set
\[
\set{\phi \in \hat \Gamma^n : \phi(e(\Gamma)) \text{ is dense in } \T^n}
\]
is dense $G_\delta$ in $\hat \Gamma^n$.

\item \label{ii:dhG:*} If $\Gamma$ is bounded and has property \eqref{eq:ast}, then the set
\[
\set{\phi \in \hat \Gamma^n : \phi(e(\Gamma)) = H_\Gamma^n}
\]
is dense $G_\delta$ in $\hat \Gamma^n$.
\end{enumerate}
\end{lemma}
\begin{proof}
In both cases it is clear that the set is $G_\delta$, so we only have to check that it is dense. We will use the fact that every character of a subgroup $\Delta \leq \Gamma$ extends to a character of $\Gamma$ (see \cite{Folland1995}*{Corollary~4.41}).

\eqref{i:dhG:ub} \textbf{Case 1.} $\Gamma$ contains an element of infinite order $\gamma_0$. Denote by $\Delta$ the cyclic subgroup generated by $\gamma_0$ and let $\theta \colon \hat \Gamma \to \hat \Delta$ be the dual of the inclusion map $\Delta \to \Gamma$. Let $U$ be an open set in $\hat \Gamma^n$. Find $\phi \in U$ such that $\theta(\phi_1), \ldots, \theta(\phi_n)$ are independent over $\Q$ (this is possible because $\theta$ is an open map and the set of rationally independent tuples is dense in $\hat \Delta^n$). Then by Kronecker's theorem, $\phi(e(\Delta))$ is dense in $\T^n$.

\textbf{Case 2.} $\Gamma$ is torsion. By Baire's theorem, it suffices to show that for any open $V \sub \T^n$, the open set of $\phi \in \hat \Gamma^n$ such that there exists $\gamma \in \Gamma$ with $\phi(e(\gamma)) \in V$ is dense.

We let $d$ denote the usual (shortest geodesic) metric on the torus. A basic open set $U$ in $\hat \Gamma^n$ has the form
\[
U = \set{\phi \in \hat \Gamma^n : \forall f \in F \ \forall i \leq n \ \phi_i(f) = \theta_i(f)},
\]
where $F$ is a finite subgroup of $\Gamma$ and $\theta_i \in \hat F$. Let also $V$ be given by
\[
V = \set{x \in \T^n : \max_i d(x_i,b_i) < \eps}
\]
for some $b \in \T^n$ and $\eps > 0$. Let $\gamma_0$ be an element of $\Gamma$ of order greater than $|F|/\eps$. Let $k$ be the least positive integer such that $\gamma_0^k \in F$ and note that $k > 1/\eps$. Let $\Delta$ be the subgroup of $\Gamma$ generated by $F$ and $\gamma_0$. For $i \leq n$, choose $c_i \in \T$ such that $d(c_i,b_i) < \eps$ and $c_i^k = \theta_i(\gamma_0^k)$ (this is possible because as $k > 1/\eps$, the preimage of any point in $\T$ under the map $x \mapsto x^k$ is $\eps$-dense) and define $\psi_i \in \hat \Delta$ by
\[
\psi_i(\gamma_0^jf) = c_i^j\theta_i(f), \quad \text{ for } j \in \Z, f \in F.
\]
($\psi_i$ is well-defined because of the choice of $c_i$.) Now extend each $\psi_i$ to an element $\phi_i \in \hat \Gamma$. Then by definition, $\phi \in U$ and $\phi(e(\gamma_0)) \in V$.

\eqref{ii:dhG:*} Let $p_1, \ldots,p_k$ denote the prime numbers dividing the orders of elements of $\Gamma$, and $n_1,\ldots,n_k$ the integers such that
$m_{\Gamma}(p_i^{n_i})=\infty$ and $m_{\Gamma}(p_i^n)=0$ for all $n>n_i$.
Pick a non-empty open  $U \sub \hat \Gamma^n$ and $x=(x_1,\ldots,x_n) \in H_{\Gamma}^n$; we need to show that there exists $\phi \in U$ such that $x \in \phi(\Gamma)$.

There exists $\psi=(\psi_1,\ldots,\psi_n) \in \hat \Gamma^n$ and a finite subgroup $\Delta \le \Gamma$ such that any element of $\hat \Gamma^n$ coinciding with $\psi$ on $\Delta$  belongs to $U$. From Pr\"ufer's structure theorem and property \eqref{eq:ast}, we see that there exists $\gamma \in \Gamma$ of order $p_1^{n_1}\cdots p_k^{n_k}$ and such that $\langle \gamma \rangle \cap \Delta= \set{1}$. The order of $x$ divides $p_1^{n_1}\cdots p_k^{n_k}$, so we may extend each $\psi_i$ to a character $\phi_i$ of $\langle \Delta, \gamma \rangle \cong \Delta \times \langle \gamma \rangle$ by setting $\phi_i(\gamma)=x_i$. Further extending $(\phi_1,\ldots,\phi_n)$ to an element of $\hat \Gamma^n$, we are done.
\end{proof}

\begin{defn} \label{df:gen-gen}
Let $n \leq \aleph_0$. A topological group $G$ is called \df{generically $n$-generated} if the set of $n$-tuples $(g_1, \ldots, g_n) \in G^n$ that generate a dense subgroup of $G$ is dense $G_\delta$ in $G^n$. $G$ is \df{generically monothetic} if it is generically $1$-generated.
\end{defn}
Note that in the definition above, the condition that the set be $G_\delta$ is automatic, so one only has to check that it is dense. Note also that if $G$ is generically $n$-generated, then it is generically $m$-generated for every $m \geq n$ and that every Polish group is generically $\aleph_0$-generated.

The idea of the proof of the next proposition comes from Kechris~\cite{Kechris2010}*{p.~26}, where by a similar argument he showed that $U(\mcH)$ is generically $2$-generated.
\begin{prop} \label{p:genericL0}
Let $\Gamma$ be a countable abelian group.
\begin{enumerate} \romanenum
\item \label{i:gen:ub} If $\Gamma$ is unbounded, then
\[
\set{\pi \in \Hom(\Gamma, L^0(\T)) : \pi(\Gamma) \text{ is dense in } L^0(\T)}
\]
is dense $G_\delta$ in $\Hom(\Gamma, L^0(\T))$. In particular, $L^0(\T)$ is generically monothetic.

\item \label{ii:gen:*} If $\Gamma$ is bounded and has property \eqref{eq:ast}, then
\[
\set{\pi \in \Hom(\Gamma, L^0(H_\Gamma)) : \pi(\Gamma) \text{ is dense in } L^0(H_\Gamma)}
\]
is dense $G_\delta$ in $\Hom(\Gamma, L^0(H_\Gamma))$.

\end{enumerate}
\end{prop}
\begin{proof}
\eqref{i:gen:ub}
Let $K_n$ denote the subgroup of $L^0([0, 1], \lambda, \T)$ consisting of all functions that are constant on intervals of the type $[i2^{-n}, (i+1)2^{-n}]$. Then $K_n$ is isomorphic to $\T^{2^n}$, the sequence $\set{K_n}$ is increasing, and $\bigcup_n K_n$ is dense in $L^0(\T)$. It is easy to check that whenever $\Gamma$ is finitely generated, $\bigcup_n \Hom(\Gamma, K_n)$ is dense in $\Hom(\Gamma, L^0(\T))$ and by using again the result about extension of characters from subgroups, we obtain that this is true for all $\Gamma$.

Thanks to the Baire category theorem, we only need to show that for any non-empty open subset $U$ of $L^0(\T)$, the open set
$$\{ \phi \in \Hom(\Gamma,L^0(\T)) : \phi(\Gamma) \cap U  \ne \emptyset\} $$
is dense in $\Hom(\Gamma,L^0(\T))$. So, pick a non-empty open subset $V$ of $\Hom(\Gamma,L^0(\T))$ and let $n$ be big enough that $K_n \cap U \ne \emptyset$ and $\Hom(\Gamma, K_n) \cap V \ne \emptyset$. Using Lemma~\ref{l:densehatGamma} \eqref{i:dhG:ub}, we know that we can find some $\phi \in \Hom(\Gamma,K_n) \cap V$ such that $\phi(\Gamma) \cap U \ne \emptyset$, which concludes the proof.

\eqref{ii:gen:*} The same proof, with obvious modifications, works.
\end{proof}


\section{Extreme amenability is a $G_\delta$ property}\label{s:extamisgdelta}
The following theorem, while simple, is the key tool for all of our results regarding extreme amenability.
\begin{theorem} \label{th:Gdelta}
Let $G$ be a Polish group and $\Gamma$ be a countable group. Then the set
\[
\set{\pi \in \Hom(\Gamma, G) : \pi(\Gamma) \text{ is extremely amenable}}
\]
is $G_\delta$ in $\Hom(\Gamma, G)$.
\end{theorem}
\begin{proof}
We begin by fixing a left-invariant metric $d$ inducing the topology of $G$. Recall from \cite{Pestov2006}*{Theorem~2.1.11} that $\pi(\Gamma)$ is extremely amenable if and only if the left-translation action of $\pi(\Gamma)$ on $(\pi(\Gamma),d)$ is finitely oscillation stable. From \cite{Pestov2006}*{Theorem~1.1.18}, we know that this is equivalent to the following condition:
\begin{multline} \label{extam}
 \forall \eps >0 \ \forall F \text{ finite} \subseteq \pi(\Gamma) \ \exists K \text{ finite} \subseteq \pi(\Gamma) \ \forall c \colon K \to \{0,1\}  \\
\exists i \in \{0,1\} \ \exists g \in \pi(\Gamma) \ \forall f \in F \ \exists k \in c^{-1}(i) \ d(g f,k) < \eps.
\end{multline}
In fact, the above property is implied by item (10) of \cite{Pestov2006}*{Theorem~1.1.18} and implies item (8) of that same theorem, which lists equivalent conditions for an action to be finitely oscillation stable.

We claim that \eqref{extam} is equivalent to the following property:
\begin{multline} \label{extam2}
\forall \eps > 0 \ \forall A \text{ finite} \subseteq \Gamma \ \exists B \text{ finite} \subseteq \Gamma \
 \forall c \colon B \to \{0,1\} \\
\exists i \in \{0,1\} \ \exists \gamma \in \Gamma \ \forall a \in A \ \exists \delta \in c^{-1}(i) \ d(\pi(\gamma a) ,\pi(\delta)) < \eps.
\end{multline}
Indeed, since any coloring of $\pi(\Gamma)$ defines a coloring of $\Gamma$, it is clear that \eqref{extam2} is stronger than \eqref{extam}. To see the converse, assume that condition \eqref{extam} is satisfied and fix $\eps>0$ and a finite subset $A$ of $\Gamma$. Then \eqref{extam} applied to $\eps$ and $F= \pi(A)$
yields a finite subset $K$ of $\pi(\Gamma)$ and we can pick a finite $B \subseteq \Gamma$ such that $\pi(B)=K$. We claim that this $B$ witnesses that condition \eqref{extam2} is satisfied. To see this, fix an enumeration $\{\gamma_n\}_{n< \omega}$ of $\Gamma$. Given $c \colon B \to \{0,1\}$, we define $c_K \colon K \to \{0,1\}$ by setting
$$c_K(k) = c(\gamma_i) \text{ where } i= \min\{j \colon \gamma_j \in B \text{ and } \pi(\gamma_j)=k\ \}.$$
To conclude, it is enough to notice that the definition of $c_K$ ensures that for any $i \in \{0,1\}$ and any $k\in c_K^{-1}(i)$, there is some $\delta \in B$ with $\pi(\delta)=k$ and $c(\delta)=i$.

It is easily checked that \eqref{extam2} is a $G_\delta$ condition on $\pi$, and we are done.
\end{proof}

Below, we will work with $\cl{\pi(\Gamma)}$ rather than $\pi(\Gamma)$. When it comes to extreme amenability the difference is immaterial since $\pi(\Gamma)$ is extremely amenable if and only if $\cl{\pi(\Gamma)}$ is extremely amenable.

\begin{remark*}
Let $\Gamma$ be a countable group. When endowed with the pointwise convergence topology, the space $D(\Gamma)$ of bi-invariant distances on $\Gamma$ is a Polish space, and
the proof above also shows that the set
\[
\set{d \in D(\Gamma) : (\Gamma, d) \text{ is extremely amenable}}
\]
is $G_\delta$ in $D(\Gamma)$. Indeed, any group endowed with a bi-invariant distance is automatically a topological group, and the reasoning above leads to a criterion for extreme amenability which is a $G_{\delta}$ condition on $d$.
\end{remark*}

\section{The unitary group}
In this section, we study generic unitary representations of abelian groups $\Gamma$, that is, generic elements of $\Hom(\Gamma, U(\mcH))$, where $\mcH$ is a separable, infinite-dimensional Hilbert space.
\begin{lemma} \label{l:densefdim}
Let $\Gamma$ be an abelian group. Then the set
\[
\set{\pi \in \Hom(\Gamma, U(\mcH)) : \pi \text{ is a direct sum of finite-dimensional representations}}
\]
is dense in $\Hom(\Gamma, U(\mcH))$.
\end{lemma}
\begin{proof}
This is a direct consequence of \cite{Kechris2010}*{Proposition~H.2}, \cite{Bekka2008}*{Proposition~F.2.7}, and the fact that irreducible representations of abelian groups are one-dimensional.
\end{proof}

Recall that a vector $\xi \in \mcH$ is \df{cyclic} for the representation $\pi$ if the span of $\set{\pi(\gamma)\xi : \gamma \in \Gamma}$ is dense in $\mcH$.
\begin{lemma} \label{l:densecyclic}
Let $\Gamma$ be an infinite abelian group and $\xi$ be a unit vector in $\mcH$. Then the set
\[
\set{\pi \in \Hom(\Gamma, U(\mcH)) : \xi \text{ is a cyclic vector for } \pi}
\]
is dense $G_\delta$ in $\Hom(\Gamma, U(\mcH))$.
\end{lemma}
\begin{proof}
We will first show that the set
\[
\set{(\pi, \xi) \in \Hom(\Gamma, U(\mcH)) \times \mcH : \xi \text{ is a cyclic vector for } \pi}
\]
is comeager. By Baire's theorem, we only need to check that for a given non-empty open set $A \sub \mcH$, the open set
\[
\set{(\pi, \xi) : \exists \gamma_1, \ldots, \gamma_n \in \Gamma \ \exists c_1, \ldots, c_n \in \C \ \sum_i c_i \pi(\gamma)\xi \in A}
\]
is dense in $\Hom(\Gamma, U(\mcH)) \times \mcH$. Fix $A$, an open set $V \sub \Hom(\Gamma, U(\mcH))$, and an open set $B \sub \mcH$. By Lemma~\ref{l:densefdim}, there is $\pi \in V$, a $\pi$-invariant finite-dimensional subspace $\mcK \sub \mcH$ and vectors $\xi \in B \cap \mcK$ and $\eta \in A \cap \mcK$. As every finite-dimensional representation of $\Gamma$ is a direct sum of characters, there exists an orthonormal basis $\set{e_i}_{i=1}^n$ of $\mcK$ and $\phi \in \hat \Gamma^n$ such that $\pi(\gamma)e_i = \ip{\gamma, \phi_i}e_i$. Let $\xi = \sum_i a_i e_i$. By slightly modifying $\xi$ if necessary, we can also assume that $a_i \neq 0$ for all $i$.
As $\Gamma$ is infinite, $\hat \Gamma$ is perfect, so without leaving $V$, we can assume that the characters $\phi_1, \ldots, \phi_n$ are distinct. By a classical theorem of Artin (see for instance \cite{Lang:Algebra}*{Theorem~4.1}), they are then linearly independent, so there exist $\gamma_1, \ldots, \gamma_n \in \Gamma$ such that $\det (\ip{\gamma_j, \phi_i})_{i, j = 1}^n \neq 0$. We thus have
\[
\det \left(a_i\ip{\gamma_j, \phi_i}\right)_{i,j=1}^n = \Big(\prod_i a_i \Big) (\det \ip{\gamma_j, \phi_i}) \neq 0,
\]
so the vectors $\pi(\gamma_j)\xi = \sum_i a_i \ip{\gamma_j, \phi_i} e_i$ are linearly independent. Every tuple of $n$ independent vectors spans $\mcK$; in particular, $\eta \in \Span \set{\pi(\gamma_j)\xi : j \leq n}$.

Now by the Kuratowski--Ulam theorem \cite{Kechris1995}*{8.41}, there exists $\xi \in \mcH$ for which the set
\[
\set{\pi \in \Hom(\Gamma, U(\mcH)) : \xi \text{ is a cyclic vector for } \pi}
\]
is comeager. By homogeneity, this is in fact true for every $\xi \neq 0$. Finally, to see that the set is $G_\delta$, it suffices to write down the definition.
\end{proof}

If $\pi \colon \Gamma \to U(\mcH)$ is a representation of the countable abelian group $\Gamma$ and $\xi \in \mcH$ is a unit vector, let $\phi_{\pi, \xi}$ be the positive-definite function corresponding to $\xi$, i.e.,
\[
\phi_{\pi, \xi}(\gamma) = \ip{\pi(\gamma)\xi, \xi}.
\]
By Bochner's theorem \cite{Folland1995}*{Theorem~4.18}, there exists a unique Borel probability measure $\mu_{\pi, \xi}$ on $\hat\Gamma$ such that
\begin{equation} \label{eq:defmu}
\int \ip{\gamma, x} \ud \mu_{\pi, \xi}(x) = \phi_{\pi, \xi}(\gamma) \quad \text{ for all } \gamma \in \Gamma.
\end{equation}
Also, the map $(\pi,\xi) \mapsto \mu_{\pi,\xi}$ is continuous.

Then the cyclic subrepresentation of $\pi$ generated by $\xi$ is isomorphic to the representation of $\Gamma$ on $L^2(\hat \Gamma, \mu_{\pi, \xi})$ given by
\begin{equation} \label{eq:isomrep}
(\gamma \cdot f)(x) = \ip{\gamma, x} f(x)
\end{equation}
with cyclic vector the constant $1$. This representation naturally extends to a representation of the $C^*$-algebra $C(\hat \Gamma)$ which we will again denote by $\pi$.

\begin{lemma} \label{l:densenonat}
Let $\Gamma$ be an infinite abelian group and $\xi$ a unit vector in $\mcH$. Then the set
\[
\set{\pi \in \Hom(\Gamma, U(\mcH)) : \xi \text{ is a cyclic vector for } \pi \text{ and }\mu_{\pi, \xi} \text{ is non-atomic}}
\]
is dense $G_\delta$ in $\Hom(\Gamma, U(\mcH))$.
\end{lemma}
\begin{proof}
Note that $\mu_{\pi, \xi}$ is non-atomic iff the cyclic subrepresentation of $\pi$ generated by $\xi$ is \df{weakly mixing}, i.e., does not contain finite-dimensional subrepresentations. Apply Lemma~\ref{l:densecyclic}, then use Proposition~H.11 and Theorem~H.12 from \cite{Kechris2010}.
\end{proof}

We are now ready to prove the main theorem of this section.
\begin{theorem} \label{th:UH}
Let $\Gamma$ be a countable abelian group.
\begin{enumerate} \romanenum
\item \label{i:t:uni}If $\Gamma$ is unbounded, then the set
\[
\set{\pi \in \Hom(\Gamma, U(\mcH)) : \cl{\pi(\Gamma)} \cong L^0(\T)}
\]
is comeager in $\Hom(\Gamma, U(\mcH))$.
\item \label{i:t:uni2}If $\Gamma$ is bounded and has property \eqref{eq:ast}, then the set
\[
\set{\pi \in \Hom(\Gamma, U(\mcH)) : \cl{\pi(\Gamma)} \cong L^0(H_\Gamma)}
\]
is comeager in $\Hom(\Gamma, U(\mcH))$.
\end{enumerate}
\end{theorem}
\begin{proof}
Again, we only give the proof for \eqref{i:t:uni}, the proof of \eqref{i:t:uni2} being similar.

Let $D$ be a countable dense subset of $C(\hat \Gamma, \T)$ (the space of all continuous functions $\hat \Gamma \to \T$ with the uniform convergence topology) and $E$ a countable dense subset of the unit sphere of $\mcH$. Define the set $C_1 \sub \Hom(\Gamma, U(\mcH))$ by
\[ \begin{split}
\pi \in C_1 \iff \forall \eps > 0 & \ \forall f \in D \ \forall \eta_1, \ldots, \eta_k \in E \\
    &\exists \gamma \in \Gamma \ \forall i \leq k \
        \nm{\pi(f)\eta_i - \pi(\gamma)\eta_i} < \eps.
\end{split} \]
We will show that $C_1$ is dense $G_\delta$. By Baire's theorem, it suffices to show that for fixed $\eps > 0$, $f \in D$, and $\eta_1, \ldots, \eta_k \in E$, the open set of $\pi$ satisfying the condition on the the second line of the definition of $C_1$ is dense. Fix an open set $V \sub \Hom(\Gamma, U(\mcH))$. By Lemma~\ref{l:densefdim}, there exist $\pi \in V$, a finite-dimensional $\pi$-invariant subspace $\mcK \sub \mcH$ and $\eta_1', \ldots, \eta_k' \in \mcK$ such that for all $i \leq k$, $\nm{\eta_i - \eta_i'} < \eps/4$ and $\nm{\eta_i'} \leq 1$. As before, there exists an orthonormal basis $e_1, \ldots, e_n$ of $\mcK$ and $\phi \in \hat \Gamma^n$ such that $\pi(\gamma)(e_j) = \ip{\gamma, \phi_j}e_j$ for all $\gamma \in \Gamma$ and $j \leq n$. By Lemma~\ref{l:densehatGamma}, we can, without leaving $V$, further assume that $\phi(\Gamma)$ is dense in $\T^n$. Find now $\gamma \in \Gamma$ such that $|\ip{\gamma, \phi_j} - f(\phi_j)| < \eps/2$ for all $j \leq n$. Then for every $\eta = \sum_j a_j e_j \in \mcK$,
\[ \begin{split}
\nm{\pi(f)\eta - \pi(\gamma)\eta}^2 &= \Big\| \sum_j a_j (f(\phi_j) - \ip{\gamma, \phi_j})e_j\Big\|^2 \\
    &\leq \sum_j |a_j|^2 (\eps/2)^2 = (\eps/2)^2 \nm{\eta}^2,
\end{split}\]
so using the choice of $\mcK$ and the fact that both $\pi(f)$ and $\pi(\gamma)$ are unitary, we obtain that $\nm{\pi(f)\eta_i - \pi(\gamma)\eta_i} < \eps$ for all $i \leq k$ as required.

Fix a unit vector $\xi \in \mcH$ and define the set $C_2 \sub \Hom(\Gamma, U(\mcH))$ by
\[
\pi \in C_2 \iff \pi \in C_1 \text{ and } \xi \text{ is cyclic for } \pi \text{ and } \mu_{\pi, \xi} \text{ is non-atomic}.
\]
By Lemma~\ref{l:densecyclic}, Lemma~\ref{l:densenonat}, and the argument above, $C_2$ is dense $G_\delta$. We now show that for every $\pi \in C_2$, $\cl{\pi(\Gamma)} \cong L^0(\T)$. Fix $\pi \in C_2$ and write $\mu$ for $\mu_{\pi, \xi}$. As $\xi$ is cyclic, $\pi$ is isomorphic to the representation on $L^2(\hat \Gamma, \mu)$ given by \eqref{eq:isomrep}, so we can identify $\mcH$ with $L^2(\hat \Gamma, \mu)$. Note that $L^0(\hat \Gamma, \mu, \T)$ embeds as a closed subgroup of $U(L^2(\hat \Gamma, \mu))$ (acting by pointwise multiplication) and that we can view $\pi(\Gamma)$ as a subgroup of $L^0(\hat \Gamma, \mu, \T)$. As $\mu$ is non-atomic, it now suffices to check that $\pi(\Gamma)$ is dense in $L^0(\hat \Gamma, \mu, \T)$. We first observe that the natural homomorphism $\theta_2 \colon C(\hat \Gamma, \T) \to L^0(\hat \Gamma, \mu, \T)$ has a dense image. Indeed, consider the commutative diagram
\[ \xymatrix{
C(\hat \Gamma, \R) \ar[r]^{\theta_1} \ar[d]^{e_1} & L^1(\hat \Gamma, \mu, \R) \ar[d]^{e_2} \\
C(\hat \Gamma, \T) \ar[r]^{\theta_2} & L^0(\hat \Gamma, \mu, \T),
} \]
where the horizontal arrows denote the quotient maps by $\mu$-a.e. equivalence and the vertical ones are the exponential maps $f \mapsto e^{if}$. As the image of $\theta_1$ is dense and $e_2$ is surjective, we obtain that the image of $\theta_2$ is also dense. Now consider an open subset of $L^0(\hat \Gamma, \mu, \T)$ of the form
\[
V = \set{g \in L^0(\hat \Gamma, \mu, \T) : \forall i \leq k \ \nm{f \cdot \eta_i - g \cdot \eta_i} < \eps}
\]
for some $f \in D$, $\eta_1, \ldots, \eta_k \in E$ and $\eps > 0$. As $L^0(\hat \Gamma, \mu, \T)$ is a closed subgroup of $U(L^2(\hat \Gamma, \mu))$ and $\theta_2(D)$ is dense in $L^0(\hat \Gamma, \mu, \T)$, sets of this form form a basis for the topology of $L^0$. As $\pi \in C_1$, each such open set meets $\pi(\Gamma)$, finishing the proof.
\end{proof}

\section{$\Aut(\mu)$}
Now, we turn to the study of generic measure-preserving actions of countable abelian groups $\Gamma$ on a standard probability space, i.e., generic elements of $\Hom(\Gamma,\Aut(\mu))$. Specifically, we prove the following theorem.

\begin{theorem}\label{t:generic-Aut(mu)}
Let $\Gamma$ be a countable abelian group which is either unbounded or bounded with property \eqref{eq:ast}. Then
\[
\set{\pi \in \Hom(\Gamma, \Aut(\mu)) : \cl{\pi(\Gamma)} \text{ is extremely amenable}}
\]
is dense $G_\delta$ in $\Hom(\Gamma, \Aut(\mu))$.
\end{theorem}

We recall that a homomorphism $\pi \in \Hom(\Gamma, \Aut(X, \mu))$ can be viewed as a measure-preserving action of $\Gamma$ on $X$. Such an action is called \df{(essentially) free} if for all $\gamma \in \Gamma \sminus \set{1}$, $\mu(\set{x \in X : \pi(\gamma) \cdot x = x}) = 0$. It is well-known that for abelian $\Gamma$, the conjugacy class of any free action of $\Gamma$ is dense (see, e.g., \cite{Foreman2004}*{Claim~19}). Since the isomorphism type of $\cl{\pi(\Gamma)}$ is invariant under conjugacy, it suffices, in view of Theorem~\ref{th:Gdelta}, to prove that there exists a single free action $\pi$ such that $\cl{\pi(\Gamma)}$ is extremely amenable. For this, we will use the Gaussian construction, which we quickly recall. As a reference for the facts stated below without proof, we refer the reader to \cite{Kechris2010}*{Appendix E} and \cite{Kechris2011}*{Lecture~4}.

Let $\nu$ denote the standard Gaussian measure on $\R$ and let $(X, \mu) = (\R^\N, \nu^\N)$. Let $p_i \colon \R^\N \to \R$, $i \in \N$ be the projection functions. Then $p_i$ are independent Gaussian random variables, $p_i \in L^2(X, \mu,\R)$, and $\ip{p_i, p_j} = \delta_{ij}$ for all $i, j$. Let $H$ be the closed subspace of $L^2(X, \mu,\R)$ spanned by the $p_i$s (known as the \df{first chaos}). For every orthogonal operator $T \in O(H)$, we can define a measure-preserving transformation $\tilde T \colon X \to X$ by
\[
\tilde T(x)(n) = (Tp_n)(x).
\]
Then $T \mapsto \tilde T$ is a topological group isomorphism between the orthogonal group $O(H)$ and a closed subgroup of $\Aut(\mu)$. We next check that each $\tilde T \neq \id$ obtained in this manner acts freely. Indeed, $T p_n - p_n$ is a Gaussian random variable for every $n$, therefore
\[
\mu(\set{x : (Tp_n)(x) = p_n(x)}) > 0 \implies Tp_n = p_n,
\]
whence $\mu(\set{\tilde T = \id }) > 0$ implies that $T = \id$.

The construction described above applies to a real Hilbert space, yet, to use our results from the previous sections, we need to embed the unitary group of a complex separable Hilbert space $\mcH$ into $\Aut(\mu)$. This is achieved by first taking the realification $\mcH_{\R}$ of $\mcH$, then embedding the unitary group of $\mcH$ into the orthogonal group of $\mcH_{\R}$ in the natural way, and then applying the Gaussian construction to $\mcH_{\R}$.

\begin{proof}[Proof of Theorem~\ref{t:generic-Aut(mu)}]
Let $\phi \colon L^0(\T) \to \Aut(\mu)$ denote the composition of the embeddings
\[
L^0(\T) \to U(\mcH) \to \Aut(\mu),
\]
where the first embedding is obtained by identifying $\mcH$ with $L^2([0, 1], \lambda)$ and letting $L^0(\T) = L^0([0, 1], \lambda, \T)$ act on it by pointwise multiplication, and the second, by the Gaussian construction outlined above. Let $\pi \colon \Gamma \to L^0(\T)$ be a homomorphism with a dense image in the unbounded case and a homomorphism whose image is dense in the copy of $L^0(H_\Gamma)$ sitting naturally inside $L^0(\T)$ in the bounded one (such a homomorphism exists by Proposition~\ref{p:genericL0}). Then $\phi \circ \pi$ is a free action of $\Gamma$ such that $\cl{(\phi \circ \pi)(\Gamma)}$ is extremely amenable, which, by the remarks above, is sufficient to prove the theorem.
\end{proof}

\begin{remark*}
We are grateful to the referee for pointing out that the actions produced via the Gaussian construction are free, thus simplifying our original proof of Theorem~\ref{t:generic-Aut(mu)}. The embedding $\phi \colon L^0(\T) \to \Aut(\mu)$ that we describe above had already been considered by Glasner and Weiss~\cite{Glasner2005a}*{p. 1537}.
\end{remark*}

\section{The Urysohn space}
We next consider generic representations of countable abelian groups by isometries of the Urysohn space, i.e., generic elements of $\Hom(\Gamma,\Iso(\bU))$.
We refer the reader to the volume \cite{TopAppl:Urysohn2008} for information about the Urysohn space $\bU$ and its bounded variant, the Urysohn sphere $\bU_1$; Chapter 5 of \cite{Pestov2006} is also a good source of information.
Let us note that all the results proved here are also true for $\Iso(\bU_1)$, and the proofs are essentially the same as those given below. To simplify the exposition, we focus on $\Iso(\bU)$.

The next lemma is a metric variant of the induction construction of Mackey and will be used in many of our arguments about representations in the isometry group of the Urysohn space.
\begin{lemma}\label{l:metric-ind}
Let $K$ be a Polish group admitting a compatible bi-invariant metric and $\Gamma \leq K$ be a discrete subgroup. Let $(Z, d_Z)$ be a bounded Polish metric space on which $\Gamma$ acts faithfully by isometries. Then there exists a bounded Polish metric space $(Y, d_Y)$, a continuous, faithful, isometric action $K \actson Y$, and an isometric, $\Gamma$-equivariant embedding $e \colon Z \to Y$.
\end{lemma}
\begin{proof}
Fix a compatible bounded bi-invariant metric $\delta$ on $K$ (note that $\delta$ is necessarily complete, see, for instance, \cite{Gao2009a}*{2.2.2}) such that
\begin{equation} \label{eq:cond-delta}
\inf_{\gamma \in \Gamma \sminus \set{1}} \delta(1, \gamma) > \diam Z.
\end{equation}
Put $X = K \times Z$ and equip it with the metric
\[
d_X((k_1, z_1), (k_2, z_2)) = \max (\delta(k_1, k_2), d_Z(z_1, z_2)).
\]
The group $\Gamma$ acts on $X$ isometrically on the right by
\[
(k, z) \cdot \gamma = (k \gamma, \gamma^{-1} \cdot z).
\]
Set $Y= X / \Gamma$ and equip it with the metric
\[ d_Y(x_1\Gamma, x_2\Gamma) = \inf_{\gamma_1, \gamma_2} d_X(x_1 \cdot \gamma_1, x_2 \cdot \gamma_2) = \inf_\gamma d_X(x_1, x_2 \cdot \gamma).
\]
Note that $d_Y$ is indeed a metric because the $\Gamma$-action is by isometries and each $x\Gamma$ is closed. It is also not difficult to see that $d_Y$ is complete.

The group $K$ acts on $Y$ on the left by
\[
k \cdot (k', z)\Gamma = (kk', z)\Gamma.
\]
This action is by isometries, and is easily seen to be continuous. Define $e \colon Z \to Y$ by $e(z) = (1, z)\Gamma$. We check that $e$ is $\Gamma$-equivariant:
\[
e(\gamma \cdot z) = (1, \gamma \cdot z)\Gamma = (\gamma, z)\Gamma = \gamma \cdot (1, z)\Gamma = \gamma \cdot e(z).
\]
We finally check that $e$ is an isometric embedding. Using \eqref{eq:cond-delta} and the fact that $\Gamma$ acts on $Z$ by isometries, we have:
\[ \begin{split}
d_Y(e(z_1), e(z_2)) &= d_Y((1, z_1)\Gamma, (1, z_2)\Gamma) \\
  &= \inf_{\gamma_1, \gamma_2} \set{d_X((\gamma_1, \gamma_1^{-1} \cdot z_1), (\gamma_2, \gamma_2^{-1} \cdot z_2))} \\
  &= \inf_{\gamma_1, \gamma_2} \set{\max(\delta(\gamma_1, \gamma_2), d_Z(\gamma_1^{-1} \cdot z_1, \gamma_2^{-1} \cdot z_2))} \\
  &= d_Z(z_1, z_2).
\end{split} \]
To see that the action is faithful, note that $k \cdot (1,z) \Gamma = (1,z)\Gamma$ implies that $k \in \Gamma$ and $k \cdot z = z$. Since the action $\Gamma \actson Z$ is assumed to be faithful, we are done.
\end{proof}
We recall for future reference that compact Polish groups, as well as abelian Polish groups, admit a compatible bi-invariant metric (see, e.g., \cite{Hewitt1979}*{8.3 and 8.9}).

For our construction below, we need to introduce some notation: whenever $(Y,d_Y)$ is a bounded Polish metric space, we consider the space $L^1(Y)$ of all (equivalence classes of) measurable maps from a standard probability space $(X,\mu)$ to $Y$, endowed with the metric
$$d(f,g)= \int_X d_Y(f(x),g(x)) \ud\mu(x).$$
Then $L^1(Y)$ is a Polish metric space in which $Y$ embeds isometrically (as constant functions). Also, if $G$ is a Polish group acting continuously on $Y$ by isometries, the action $G \actson Y$ extends to an isometric action $L^0(G) \actson L^1(Y)$ defined by
$$(g \cdot f)(x)=g(x)\cdot f(x). $$
This action is continuous, faithful if the original action $G \actson Y$ is faithful, and the image of $L^0(G)$ in the isometry group of $L^1(Y)$ is closed if $G$ is a closed subgroup of $\Iso(Y)$.

\begin{lemma} \label{l:L0denseU}
\begin{enumerate} \romanenum
\item \label{i:l:L0d}Let $\Gamma$ be an unbounded abelian group. Then the set
\[
\set{\pi \in \Hom(\Gamma, \Iso(\bU)) : \cl{\pi(\Gamma)} \cong L^0(\T)}
\]
is dense in $\Hom(\Gamma, \Iso(\bU))$.
\item \label{ii:l:L0d} Let $\Gamma$ be a bounded abelian group with property \eqref{eq:ast}. Then the set
\[
\set{\pi \in \Hom(\Gamma, \Iso(\bU)) : \cl{\pi(\Gamma)} \cong L^0(H_{\Gamma})}
\]
is dense in $\Hom(\Gamma, \Iso(\bU))$.
\end{enumerate}
\end{lemma}
\begin{proof}
\eqref{i:l:L0d}
We denote $G=\Iso(\bU)$. Note that open sets in $\Hom(\Gamma, G)$ of the form
\[
\set{\pi : d(\pi(\gamma)a, \pi_0(\gamma)a) < \eps \text{ for all } \gamma \in F, a \in A},
\]
where $\pi_0 \in \Hom(\Gamma, G)$, $\eps > 0$, $F$ is a finite subset of $\Gamma$, and $A$ is a finite subset of $\bU$, form a basis for the topology of $\Hom(\Gamma, G)$. We will show that every such open set contains a $\pi$ such that $\cl{\pi(\Gamma)} \cong L^0(\T)$.

To that end, fix $\pi_0$, $\eps$, $F$, and $A$. Denote by $\Delta$ the subgroup of $\Gamma$ generated by $F$. As $\Delta$ is residually finite, we may apply \cite{Pestov2006a}*{Theorem~1} to find an action $\pi_1 \colon \Delta \actson \bU$ such that $d(\pi_1(\gamma)a,\pi_0(\gamma)a)< \eps$ for all $\gamma \in F$, $a \in A$ and such that $\pi_1$ has finite range in $\Iso(\bU)$. Let $B$ denote the (finite) closure of $A$ under $\pi_1$, and $\Delta'$ be the image of $\Delta$ in $\Iso(B)$. Find an integer $n$ such that $\Delta'$ embeds as a subgroup of $\T^n$. By Lemma~\ref{l:metric-ind}, there exists a bounded Polish metric space $Y$ and a faithful action of $\T^n$ on $Y$ which extends the action $\Delta' \actson B$. The natural homomorphism $\Delta \to \T^n$ extends to a homomorphism $\Gamma \to \T^n$ which we denote by $\phi$. The action $\T^n \actson Y$ gives rise, as explained above, to an action $L^0(\T^n) \actson L^1(Y)$ which extends the action $\T^n \actson Y$, and we may identify $L^0(\T^n)$ with its image in $\Iso(L^1(Y))$. By the Kat\v{e}tov construction (see, e.g., \cite{Pestov2006}*{pp. 109--111, 115}), the action $\Iso(L^1(Y)) \actson L^1(Y)$ extends to an action $\Iso(L^1(Y)) \actson \bU$, and
$\Iso(L^1 (Y))$ embeds as a closed subgroup of $G$. Thus we have identified $L^0(\T^n)$ with a closed subgroup of $G$.

The composition of the homomorphisms
\[
\Gamma \to \T^n \to L^0(\T^n) \to G
\]
defines an element $\pi$ of $\Hom(\Gamma, G)$ and we can assume, using the homogeneity of $\bU$, that for all $\gamma \in \Delta$, one has $\pi(\gamma)|_A = \pi_1(\gamma)|_A$, hence $\pi \in U$. Now by perturbing $\pi$ slightly, using Proposition~\ref{p:genericL0}, we can arrange so that $\pi$ is still in $U$ but $\pi(\Gamma)$ is dense in $L^0(\T^n)$. As $L^0(\T^n)$ is closed in $G$, we obtain that $\cl{\pi(\Gamma)} = L^0(\T^n) \cong L^0(\T)$, finishing the proof.

\eqref{ii:l:L0d} Essentially the same proof works (replacing everywhere $\T$ by $H_\Gamma$).
\end{proof}

The next theorem follows immediately from Lemma~\ref{l:L0denseU} and Theorem~\ref{th:Gdelta}.
\begin{theorem} \label{t:generic-iso(U)}
Let $\Gamma$ be a countable abelian group, and assume that $\Gamma$ is either unbounded or is bounded and has property \eqref{eq:ast}. Then
\[
\set{\pi \in \Hom(\Gamma, \Iso(\bU)) : \cl{\pi(\Gamma)} \text{ is extremely amenable}}
\]
is dense $G_\delta$ in $\Hom(\Gamma, \Iso(\bU))$.
\end{theorem}

\begin{remark*}
We note that, in view of Proposition~\ref{p:discon}, Theorem~\ref{t:generic-Aut(mu)} and Theorem~\ref{t:generic-iso(U)} are best possible, in the sense that the condition on $\Gamma$ cannot be replaced with a more general one.
\end{remark*}

Recall that, whenever $\Gamma$ is a countable group, the space $D(\Gamma)$ of all bi-invariant distances on $\Gamma$ endowed with the pointwise convergence topology is a Polish space.

\begin{theorem} \label{t:extamisgdelta}
Let $\Gamma$ be a countable abelian group which is either unbounded or is bounded and has property \eqref{eq:ast}. Then the set
$$\{\rho \in D(\Gamma) : (\Gamma, \rho) \text{ is extremely amenable}\} $$
is dense $G_{\delta}$ in $D(\Gamma)$.
\end{theorem}

\begin{proof}
Given the remark at the end of Section~\ref{s:extamisgdelta}, it is enough to prove that the set of $\rho$ for which $(\Gamma, \rho)$ is extremely amenable is dense in $D(\Gamma)$.

Let $\rho_0$ denote an invariant distance on $\Gamma$. By the Kat\v{e}tov construction, there exists a metric embedding $j \colon (\Gamma, \rho_0) \to \bU$ and a homomorphism $\pi \colon \Gamma \to \Iso(\bU)$ such that the left translation action $\Gamma \actson (\Gamma, \rho_0)$ and $\Gamma \actson^\pi j(\Gamma)$ are isomorphic. Fix a finite subset $A$ of $\Gamma$ and $\eps > 0$, and let $O$ be the open neighborhood of $\rho_0$ in $D(\Gamma)$ consisting of all $\rho$ such that
$$\forall \gamma \in A \ |\rho(\gamma,1_{\Gamma}) - \rho_0(\gamma,1_{\Gamma})| < \eps. $$
By Theorem~\ref{t:generic-iso(U)}, we know that there exists $\tilde \pi \in \Hom(\Gamma, \Iso(\bU))$ such that $\tilde \pi(\Gamma)$ is extremely amenable, $\tilde \pi$ is injective (since the generic $\tilde \pi$ is injective), and
$$ \forall \gamma \in A \ d(\tilde \pi(\gamma)(j(1_{\Gamma})),\pi(\gamma)(j(1_{\Gamma}))) < \eps. $$
Let $\sigma$ denote a left-invariant distance on $\Iso(\bU)$ compatible with its topology and note that, as $\Gamma$ is abelian, $\sigma|_{\tilde \pi(\Gamma)}$ is bi-invariant. For any integer $N>0$, one can consider the bi-invariant distance $\rho_N$ on $\Gamma$ defined by
\[ \begin{split}
\rho_N(\gamma_1,\gamma_2) &= d(\tilde \pi(\gamma_1)(j(1_{\Gamma})), \tilde \pi(\gamma_2)(j(1_{\Gamma}))) + \frac{1}{N}\sigma(\tilde \pi(\gamma_1), \tilde \pi(\gamma_2)) \\
&= \rho_0(\gamma_1, \gamma_2) + \frac{1}{N}\sigma(\tilde \pi(\gamma_1), \tilde \pi(\gamma_2)).
\end{split} \]
Then $\rho_N$ belongs to $O$ for $N$ large enough; also, $\tilde \pi \colon (\Gamma, \rho_N) \to (\tilde \pi(\Gamma), \sigma)$ is a uniform isomorphism, whence $(\Gamma, \rho_N)$ is isomorphic to $\tilde \pi(\Gamma)$ as a topological group and therefore extremely amenable.
\end{proof}

\section{The extreme amenability of $U(\mcH)$, $\Aut(\mu)$ and $\Iso(\bU)$} \label{s:free-groups}
In this section, we describe how to deduce the extreme amenability of $U(\mcH)$, $\Aut(\mu)$, and $\Iso(\bU)$ from the extreme amenability of $L^0(K)$, where $K$ is compact (see \cite{Glasner1998a} or \cite{Pestov2006}), and Theorem~\ref{th:Gdelta}. Even though the extreme amenability of those groups is known, we think that our ``uniform'' proof is interesting in its own right.

We recall that any Polish group is generically $\aleph_0$-generated; it is then clear that the following theorem, along with Theorem~\ref{th:Gdelta} and the fact that $L^0(U(m))$ is extremely amenable for all $m$, implies the extreme amenability of our three groups.
\begin{theorem}
Let $G$ be either of $U(\mcH)$, $\Aut(\mu)$, $\Iso(\bU)$. Then the set
\[
\set{\pi \in \Hom(\F_\infty, G) : \exists m \ \cl{\pi(\F_\infty)} \cong L^0(U(m))}
\]
is dense in $\Hom(\F_\infty, G)$.
\end{theorem}
\begin{proof}
It suffices to show that for every $n$,
\[
\set{(g_1, \ldots, g_n) \in G^n : \exists m \ \langle g_1, \ldots, g_n \rangle \text{ is contained in a closed copy of } L^0(U(m))}
\]
is dense. Indeed, any open set in $\Hom(\F_\infty, G)$ is specified by imposing conditions on finitely many generators, say $a_1, \ldots, a_n$; if we can arrange that $\pi(a_1), \ldots, \pi(a_n)$ are contained in a closed copy of $L^0(U(m))$, we can use the other generators to ensure that $\cl{\pi(\F_\infty)} \cong L^0(U(m))$.

\smallskip
\emph{The unitary group}. We fix an orthonormal basis $\set{e_i}$ of $\mcH$ and let $\mcH(m) = \Span(e_1, \ldots, e_m)$. Identify the finite-dimensional unitary group $U(m)$ with the subgroup of $U(\mcH)$ that acts trivially on $\mcH(m)^\perp$. Fix a non-empty open subset $V$ of $U(\mcH)^n$. Find $m$ sufficiently large and a tuple $(g_1, \ldots, g_n) \in V \cap U(m)^n$ such that whenever $(h_1, \ldots, h_n) \in U(\mcH)^n$ agrees with $(g_1, \ldots, g_n)$ on $\mcH(m)$, then $(h_1, \ldots, h_n) \in V$.

The action $U(m) \actson \mcH(m)$ extends naturally to an action
\[
L^0(U(m)) = L^0([0,1], \lambda, U(m)) \actson L^2([0,1], \lambda, \mcH(m)).
\]
This gives a homeomorphic embedding of $L^0(U(m))$ as a closed subgroup of the unitary group of $L^2([0,1], \lambda, \mcH(m))$. Denote by $e \colon U(m) \to U(L^2([0,1], \lambda, \mcH(m)))$ the composition of the embeddings $U(m) \to L^0(U(m)) \to U(L^2([0,1], \lambda, \mcH(m)))$. Let $T \colon L^2([0,1], \lambda, \mcH(m)) \to \mcH$ be a unitary isomorphism that is the identity on $\mcH(m)$ (we see $\mcH(m)$ as a subspace of $L^2([0,1], \lambda, \mcH(m))$ by identifying it with the subspace of constant functions), so that $T$ conjugates the two actions of $U(m)$ on $\mcH(m)$. Finally, let $(h_1, \ldots, h_n) = T e((g_1, \ldots, g_n)) T^{-1}$ and observe that $(h_1, \ldots, h_n) \in V$ and $h_1, \ldots, h_n$ are contained in the subgroup $T(L^0(U(m))) \leq U(\mcH)$.

\smallskip

\emph{The isometry group of the Urysohn space}. In that case, the proof is basically the same as that of Lemma~\ref{l:L0denseU}. Pick a non-empty open subset $V$ of $\Iso(\bU)^n$, and use the fact that $\bF_n$ is residually finite and \cite{Pestov2006a}*{Theorem~1} to find $(g_1,\ldots,g_n)$ belonging to $V$ and generating a finite subgroup $\Delta$. Then pick a finite subset $A$ of $\bU$ such that $\Delta \actson A$ and any tuple coinciding with $(g_1,\ldots,g_n)$ on $A$ belongs to $V$; let $\Delta'$ denote the image of $\Delta$ in $\Iso(A)$. Let $m = |\Delta'|$; then $\Delta'$ embeds in the unitary group $U(m)$ (by its left-regular representation). Using Lemma~\ref{l:metric-ind}, the action $\Delta' \actson A$ extends to a faithful action $U(m) \actson Y$ for some bounded Polish metric space $Y$, which then extends to an action $L^0(U(m)) \actson L^1(Y)$, and $L^0(U(m))$ is isomorphic to the corresponding subgroup of $\Iso(L^1(Y))$.
Applying the Kat\v{e}tov construction and using the homogeneity of $\bU$, we obtain $(h_1,\ldots,h_n) \in V$ such that the subgroup generated by $(h_1,\ldots,h_n)$ is contained in an isomorphic copy of $L^0(U(m))$.

\smallskip

\emph{The automorphism group of $(X,\mu)$}.
Call an action $\F_n \actson (X, \mu)$ \df{finite} if it factors through the action of a finite group. It is an immediate consequence of the fact that $\Aut(\mu)$ has a locally finite dense subgroup that the finite actions of $\F_n$ are dense. Fix now an open neighborhood $V$ of a finite action $\F_n \actson^{\phi_0} (X,\mu)$ that factors through the finite group $\Delta$. As free actions of $\Delta$ are dense in $\Hom(\Delta, \Aut(\mu))$, we can assume that the action $\psi_0$ of $\Delta$ induced by $\phi_0$ is free. Let $m$ be the order of $\Delta$ and, as above, identify $\Delta$ with a subgroup of $U(m)$ via the left-regular representation. Let $(Y, \nu)$ be a standard probability space and consider the measure-preserving action $L^0(Y, \nu, U(m)) \actson^a Y \times U(m)$ defined by
\[
f \cdot (y, u) = (y, f(y)u)
\]
and observe that this action gives an embedding of $L^0(Y, \nu, U(m))$ as a closed subgroup of $\Aut(Y \times U(m), \nu \times \text{Haar})$, which we identify with $\Aut(\mu)$. As $a|_{\Delta}$ is free and any two free actions of $\Delta$ are conjugate, we can assume that $\psi_0 = a|_{\Delta}$.
Denoting by $\phi$ the action of $\F_n$ induced by $a|_{\Delta}$, we thus have that $\phi \in V$ and $\phi(\F_n)$ is contained in an isomorphic copy of $L^0(U(m))$.
\end{proof}

\begin{remark*}
In fact, each of these three groups is generically $2$-generated: for $U(\mcH)$ this is proved on page 26 of \cite{Kechris2010}; for $\Aut(\mu)$ it is a result of Prasad~\cite{Prasad1981}; and for $\Iso(\bU)$, this is a recent result of Slutsky~\cite{Slutsky2011}*{4.19}. Hence, for the above proof, it would be enough to only consider pairs of elements of $G$. If one wants to follow that route, one also needs the fact that $L^0(U(m))$ is generically $2$-generated; actually, $L^0(K)$ is generically $2$-generated whenever $K$ is compact, connected and metrizable, a fact that can be proved using the arguments of Proposition \ref{p:genericL0} and a classical result of Schreier and Ulam~\cite{Schreier1933} stating that any compact, connected and metrizable topological group is generically $2$-generated.
\end{remark*}

\section{Generic properties of homomorphisms from torsion-free abelian groups into $\Aut(\mu)$ and $\Iso(\bU)$}
We now turn to the question: how much do generic properties of $\cl{\pi(\Gamma)}$ depend on the (abelian) group $\Gamma$? Of course, if $\Gamma$ is bounded, then the group $\cl{\pi(\Gamma)}$ is also bounded and, in general, it will be different from a generic $\cl{\pi(\Z)}$. However, as it turns out, if the target group is $\Aut(\mu)$ or $\Iso(\bU)$, and one restricts attention to free abelian groups, the generic properties of $\cl{\pi(\Z^d)}$ do not depend on $d$. As another application of our techniques (together with some Baire category tools developed in Appendix~\ref{a:category}), we also identify exactly the centralizers of generic $\pi(\Gamma)$ for torsion-free groups $\Gamma$.

\subsection{A criterion for monotheticity}
We now prove a simple criterion that will allow us to show that the generic $\cl{\pi(\Gamma)}$ is monothetic (of course, $\cl{\pi(\Z)}$ is always monothetic but this is not a priori clear for other groups $\Gamma$).
\begin{prop} \label{p:criterion-monothetic}
Let $G$ be a Polish group, $\Gamma$ a countable group. Suppose that for every $\delta \in \Gamma$, $\gamma \in \Gamma \sminus \set{1}$, the set
\[
A_{\gamma, \delta} = \set{\pi \in \Hom(\Gamma, G) : \pi(\delta) \in \langle \pi(\gamma) \rangle}
\]
is dense. Then both sets
\[
B_1 = \set{\pi \in \Hom(\Gamma, G) : \cl{\pi(\Gamma)} \text{ is generically monothetic}}
\]
and
\[
B_2 = \set{\pi \in \Hom(\Gamma, G) : \langle \pi(\gamma) \rangle \text{ is dense in }\cl{\pi(\Gamma)} \text{ for every } \gamma \in \Gamma \sminus \set{1} }
\]
are dense $G_\delta$.
\end{prop}
\begin{proof}
First we check that $B_1$ is $G_\delta$. Let $\set{U_n}$ be a basis for the topology of $G$. We have
\begin{multline*}
\pi \in B_1 \iff \Big( \forall m, n \ \big(\pi(\Gamma) \cap U_m \ne \emptyset \text{ and } \pi(\Gamma) \cap U_n \ne \emptyset \big) \\
\implies \big(\exists \gamma \in \Gamma, k \in \Z \ \pi(\gamma) \in U_m \text{ and } \pi(\gamma^k) \in U_n \big) \Big).
\end{multline*}
Indeed, if $\cl{\pi(\Gamma)}$ is generically monothetic, then the set of topological generators is dense; taking $U_n, U_m$ as above, there exists a topological generator
$g$ of $\cl{\pi(\Gamma)}$ belonging to $U_m$, and some $g^k$ must belong to $U_n$. Since $g$ is a limit of elements of $\pi(\Gamma)$, one can find $\gamma$ such that
$\pi(\gamma) \in U_m$ and $\pi(\gamma^k) \in U_n$. Conversely, if the condition on the right-hand side of the equivalence above is satisfied, then for each open non-empty subset $U$ of $\cl{\pi(\Gamma)}$, the set of all $g \in \cl{\pi(\Gamma)}$ such that $\langle g \rangle \cap U \ne \emptyset$ is open and dense in $\cl{\pi(\Gamma)}$; the intersection of all these sets as $U$ ranges over a countable basis for the topology of $\cl{\pi(\Gamma)}$ is the set of all topological generators of $\cl{\pi(\Gamma)}$, which is thus dense $G_{\delta}$.

Then, we show that  $B_{\gamma,\delta} = \set{\pi \in \Hom(\Gamma, G) : \pi(\delta) \in \cl{\langle \pi(\gamma) \rangle}}$ is $G_{\delta}$.
As $B_2$ is the intersection of all $B_{\gamma,\delta}$ ($\gamma \in \Gamma \sminus \{1\}$, $\delta \in \Gamma$), this will show that $B_2$ is $G_{\delta}$. We have
\[
\pi \in B_{\gamma,\delta} \iff \Big( \forall n  \big(\pi(\delta) \in U_n  \implies \exists k \in \Z \ \pi(\gamma^k) \in U_n \big) \Big).
\]
Finally, as $B_2 \sub B_1$, it suffices to check that $B_2$ is dense. This is a direct consequence of Baire's theorem and the fact that $A_{\gamma,\delta} \sub B_{\gamma,\delta}$.
\end{proof}

The following theorem is the main technical tool for our applications.
\begin{theorem}\label{t:generic-monothetic}
Let $G$ be one of $\Aut(\mu)$ and $\Iso(\bU)$, and let $\Gamma$ be a countable, torsion-free, abelian group. Then the sets
\[
\set{\pi \in \Hom(\Gamma, G) \colon \cl{\pi(\Gamma)} \text{ is generically monothetic}}
\]
and
\[
\set{\pi \in \Hom(\Gamma,G) \colon \forall \gamma \in \Gamma \setminus \{1\} \ \cl{\langle \pi(\gamma) \rangle }=\cl{\pi(\Gamma)}}
\]
are dense $G_\delta$ in $\Hom(\Gamma, G)$.
\end{theorem}
\begin{proof}
We show that both $\Aut(\mu)$ and $\Iso(\bU)$ satisfy the hypothesis of Proposition~\ref{p:criterion-monothetic}. In what follows, denote by $e_1, \ldots, e_d$ be the standard generators of $\Z^d$.

We begin with the case $G = \Aut(\mu)$; the following is a consequence of the Rokhlin lemma for free abelian groups \cite{Conze1972}*{Theorem~3.1}.
\begin{lemma}\label{l:primeorder-measure}
Let $\eps > 0$. Then for every free $\pi \in \Hom(\Z^d, \Aut(\mu))$ and all $n_1,\ldots,n_d \in \N$, there exists $\psi \in \Hom(\Z^d, \Aut(\mu))$ such that for all $i$,
\begin{itemize}
\item $\psi(e_i)^{n_i}=1$;
\item $\mu\big(\{x \colon \pi(e_i)(x) \ne \psi(e_i)(x) \}\big) \le \eps + 1/n_i$.
\end{itemize}
\end{lemma}

For $\mcA=\{A_1,\ldots,A_n\}$ a finite measurable partition of $X$ and $g,h \in G$, we use the notation $d(g\mcA, h \mcA)< \eps$ as shorthand for the more cumbersome $\forall i \leq n \ \mu(g(A_i) \sdiff h(A_i)) < \eps$.

Fix $\gamma, \delta \in \Gamma$, $\gamma \neq 1$ and let $U$ be a non-empty open set in $\Hom(\Gamma, \Aut(\mu))$.
We can assume that $U$ is given by
\[
\pi \in U \iff d(\pi(e_i)\mcA, \pi_0(e_i)\mcA) < \eps \quad \text{for } i = 1, \ldots, k,
\]
where $\eps >0$, $\pi_0 \in \Hom(\Gamma,\Aut(\mu))$, $\mcA$ is a finite measurable partition of $X$, $e_1, \ldots, e_k$ are independent elements of $\Gamma$ (i.e., $\langle e_1, \ldots, e_k \rangle \cong \Z^k$) and $\gamma, \delta \in \langle e_1, \ldots, e_k \rangle$. By changing the basis $e_1, \ldots, e_k$ if necessary, we can further assume that $\gamma = e_1^{s_1} \cdots e_k^{s_k}$ and $s_i \neq 0$ for all $i$. Put $\Delta = \langle e_1, \ldots, e_k \rangle$.

By Lemma~\ref{l:primeorder-measure}, there exist distinct prime numbers $p_1, \ldots, p_k$ with $\min(p_i) > \max (s_i)$ such that if $\Delta' = \Z(p_1) \times \cdots \times \Z(p_k)$ and $q \colon \Delta \to \Delta'$ denotes the natural quotient map, then there exists an action $\Delta' \actson^\psi X$ such that for all $i$, $d((\psi \circ q)(e_i)\mcA, \pi_0(e_i)\mcA) < \eps$. Note also that $\Delta'$ is cyclic and $q(\gamma)$ is a generator. Consider now $\Delta'$ as a subgroup of $\T$, so that $q$ becomes a character of $\Delta$. Then $q$ extends to a homomorphism $\Gamma \to \T$ (still denoted by $q$); denote its image by $\Gamma'$.

Let $\Gamma' \actson^{\psi'} X^{\Gamma'/\Delta'}$ be the action of $\Gamma'$ co-induced from $\psi$ (see, e.g., \cite{Kechris2010}*{10 (G)} for details on co-induction) and let $\Theta \colon X^{\Gamma' / \Delta'} \to X$ be the factor map, $\Theta(\omega) = \omega(\Delta')$, which is $\Delta'$-equivariant. Denote by $\mcB$ the finite measurable partition generated by $\mcA$ and $\psi(\mcA)$, let $T \colon X^{\Gamma'/\Delta'} \to X$ be any measure-preserving isomorphism such that $T^{-1}(B) = \Theta^{-1}(B)$ for all $B \in \mcB$, and define $\pi \colon \Gamma \to \Aut(\mu)$ by $\pi = q \circ (T\psi'T^{-1})$. It is clear that $\pi \in U$ and $\pi(\delta) \in \langle \pi(\gamma) \rangle$ (as $\pi(\gamma)$ is a generator of $\pi(\Delta)$). This concludes the proof for $G= \Aut(\mu)$.

For $\Iso(\bU)$, we adapt an argument of Pestov and Uspenskij~\cite{Pestov2006a}.
\begin{lemma}\label{l:easymodif1}
Let $n$ be a non-negative integer, $\Lambda=\Z^d * \F_n$ and $C$ be a finite subset of $\Lambda$. Then there exists an integer $M$ such that for any $p_1,\ldots,p_d \ge M$, there exists a finite group $L$ and a morphism $q \colon \Lambda \to L$ such that $q|_C$ is injective and $q(e_i)$ has order $p_i$ for all $i \leq d$.

\end{lemma}
\begin{proof}
Let $a_1, \ldots, a_n$ be generators of $\F_n$. Consider each element of $C$ as a reduced word on the letters $e_1, \ldots, e_d, a_1, \ldots, a_n$ and let $N$ denote the greatest integer such that a subword of the form $e_i^{\pm N}$ appears in some element of $C$. Then, pick any $p_i$ greater than $M = 2N+1$. Let $q_0 \colon \Lambda \to ( \prod_{i=1}^d  \Z(p_i) ) * \F_n$ be the homomorphism defined by extending the canonical projection $\Z^d \to \prod_{i=1}^d  \Z(p_i)$. Note that the restriction $q_0|_C$ is one-to-one, and each $q_0(e_i)$ has order $p_i$. Now let $C'=q_0(\Z^d) \cup q_0(C)$. Since free products of residually finite groups are residually finite (see \cite{Gruenberg1957}, or \cite{Pestov2006}*{5.3.5}) and $C'$ is finite, there exists a finite group $L$ and a morphism $q_1 \colon ( \prod_{i=1}^d  \Z(p_i) ) * \F_n \to L$
such that the restriction $q_1|_{C'}$ is injective. Setting $q = q_1 \circ q_0$, the lemma is proved.
\end{proof}

With this new ingredient, we obtain the following variant on Lemma~4 of \cite{Pestov2006a}.
\begin{lemma}\emph{(variation on Pestov--Uspenskij \cite{Pestov2006a})} \label{l:easymodif2}
Let $\Lambda=\Z^d * \F_n$, and $K$ be a finite subset of $\Lambda$. Fix a left-invariant pseudometric on $\Lambda$ whose restriction to $K$ is a metric. Then there exists $M \in \N$ such that for any $p_1,\ldots,p_d \ge M$ there exists a finite group $L$, a morphism $q \colon \Lambda \to L$ and a left-invariant metric on $L$ such that:
\begin{itemize}
 \item the restriction of $q$ to $K$ is distance-preserving;
 \item $q(e_i)$ has order $p_i$.
\end{itemize}
\end{lemma}

\begin{proof}
The proof is exactly the same as in \cite{Pestov2006a}, using Lemma~\ref{l:easymodif1} above instead of simply using the fact that $\Lambda$ is residually finite.
\end{proof}

Redoing the proof of Theorem~1 in \cite{Pestov2006a}, with Lemma~\ref{l:easymodif2}
replacing their Lemma~4, we obtain the following result.
\begin{theorem}[variation on Pestov--Uspenskij \cite{Pestov2006a}*{Theorem~1}] \label{l:primeorder-metric}
Let $G= \Iso(\bU)$ and $U$ be a non-empty open subset of $\Hom(\Z^d,G)$. Then there exists $M$ such that for any $p_1,\ldots ,p_d \ge M$ there exists $\pi \in U$ such that
$\pi(e_i)$ has order $p_i$ for all $i$.
\end{theorem}

We now proceed to show that the hypothesis of Proposition~\ref{p:criterion-monothetic} is satisfied if $\Gamma$ is a countable, torsion-free, abelian group and $G=\Iso(\bU)$.
Let $U$ be a non-empty open subset  of $\Hom(\Gamma,G)$, let $\delta \in \Gamma$ and $\gamma \in \Gamma \setminus\{1\}$, and assume without loss of generality that there exist a finite $A \subseteq \bU$, $\eps >0$, $\pi_0 \in \Hom(\Gamma,G)$ and $e_1,\ldots,e_k \in \Gamma$ such that
\[
\pi \in U \iff \forall a \in A \ d \big(\pi(e_i)(a),\pi_0(e_i)(a) \big) < \eps.
\]
As above, we also assume that $e_1,\ldots,e_k$ are independent, that $\delta, \gamma \in \Delta = \langle e_1,\ldots,e_k \rangle$ and that $\gamma = e_1^{s_1} \cdots e_k^{s_k}$ with $s_i \ne 0$ for all $i$. Applying Theorem~\ref{l:primeorder-metric}, we can find distinct prime numbers $p_1,\ldots,p_k$ such that $\min p_i > \max s_i$ and, letting $\Delta'$ denote $\Z(p_1) \times \cdots \times \Z(p_k)$ and $q \colon \Delta \to \Delta'$ the natural quotient homomorphism, an action  $\Delta' \actson^{\psi} \bU$ such that
\[
\forall a \in A \ \forall i \leq k \  d \big(\psi \circ q(e_i)(a),\pi_0(e_i)(a) \big) < \eps.
\]
Note that, since $q(\gamma)$ is a generator of $\Delta'$, one has that $\psi \circ q (\delta) \in \langle \psi \circ q(\gamma) \rangle$. Now, $B = \psi(\Delta') \cdot A$ is finite; applying Lemma~\ref{l:metric-ind}, we can extend $\psi \circ q \colon \Delta \actson B$ to $\tilde \psi \colon \Gamma \actson X$, where $X$ is a countable metric space, and the construction ensures that we still have $\tilde \psi(\delta) \in \langle \tilde \psi(\gamma) \rangle$. Use the Kat\v{e}tov construction to extend $\tilde \psi$ to an isometric action, still denoted $\tilde \psi$, of $\Gamma$ on $\bU$. We still have $\tilde \psi (\delta) \in \langle \tilde \psi(\gamma) \rangle$ and, using the homogeneity of $\bU$, we may assume that
\[
\forall a \in A \ \forall i \leq k \ \tilde \psi(e_i)(a)= \psi \circ q (e_i)(a).
\]
Hence $\tilde \psi \in U$, and we are done.
\end{proof}

\subsection{Applications}\label{applications}
In this subsection, we derive some consequences of Theorem~\ref{t:generic-monothetic}. To that end, we need a generalization of the classical Kuratowski--Ulam theorem, which we discuss in Appendix~\ref{a:category}. For the moment, let us simply state the definition of a category-preserving map, which we use below: a continuous map between Polish spaces $f \colon X \to Y$ is \df{category-preserving} if $f^{-1}(A)$ is comeager in $X$ whenever $A$ is comeager in $Y$; equivalently, $f$ is category-preserving if it is continuous and $f^{-1}(A)$ is dense whenever $A$ is open and dense in $Y$. A coordinate projection from $X \times X$ to $X$ is an example of a category-preserving map, as is any open map; the Kuratowski--Ulam theorem, usually stated for projections, extends to all category-preserving maps (see Theorem~\ref{t:KU2}).

Category-preserving maps have already been considered in the ergodic theory literature: first by King~\cite{King2000} and later by Ageev~\cite{Ageev2000}, Tikhonov~\cite{Tikhonov2003}, and Stepin--Eremenko~\cite{Stepin2004}. The definition used in these papers is different from ours but equivalent to it (see \cite{Melleray2012}, where this matter is discussed in some detail).


The following lemma provides, in a rather special situation, a simple criterion for the restriction of homomorphisms to be a category-preserving map.
\begin{lemma} \label{l:commuting-categorypreserving}
Let $G$ be a Polish group and $\Gamma$ a countable, abelian group. Assume that for a dense set of
$\pi \in \Hom(\Gamma \times \Z, G)$, $\cl{\pi(\Gamma \times \Z)} = \cl{\pi(\Gamma)}$. Then the restriction $\Res \colon \Hom(\Gamma \times \Z, G) \to \Hom(\Gamma, G)$ is category-preserving.
\end{lemma}
\begin{proof}
Let $A$ be a dense subset of $\Hom(\Gamma, G)$; we will show that $\Res^{-1}(A)$ is dense in $\Hom(\Gamma \times \Z, G)$. Let $O$ be an open, non-empty subset of $\Hom(\Gamma \times \Z, G)$; we can assume that it is of the form
\[
O = \set{\pi \in \Hom(\Gamma \times \Z, G) : \Res(\pi) \in U \text{ and } \pi(g) \in V},
\]
where $g \in \Gamma \times \Z$ is a generator of $\Z$, $U$ is an open subset of $\Hom(\Gamma, G)$, and $V$ is an open subset of $G$. Let $\pi_0 \in O$ be such that $\cl{\pi_0(\Gamma \times \Z)} = \cl{\pi_0(\Gamma)}$. Then there exists $\gamma_0 \in \Gamma$ with $\pi_0(\gamma_0) \in V$. By shrinking $U$ if necessary, we can assume that $\psi \in U \implies \psi(\gamma_0) \in V$. Let now $\psi \in U \cap A$ be arbitrary and define $\pi \colon \Gamma \times \Z \to G$ by $\pi(\gamma,n) = \psi(\gamma) \psi(\gamma_0)^n$. Then clearly $\Res(\pi) \in A$ and $\pi \in O$.
\end{proof}

For the next results, we see $\Z$ embedded in $\Z^d$ as the subgroup generated by $e_1$.
\begin{theorem}\label{t:res-cat1}
Let $G$ be either $\Aut(\mu)$ or $\Iso(\bU)$. Then the restriction map
\[
\Res \colon \Hom(\Z^d,G) \to \Hom(\Z, G)
\]
is category-preserving.
\end{theorem}
\begin{proof}
Since the composition of category-preserving maps is category-preserving, the theorem follows from applying Theorem~\ref{t:generic-monothetic} and Lemma~\ref{l:commuting-categorypreserving} to the sequence of restriction maps
\[
\Hom(\Z^{d}, G) \to \Hom(\Z^{d-1}, G) \to \cdots \to \Hom(\Z, G). \qedhere
\]
\end{proof}
\begin{remark*}
After this paper had been written, we found out that Theorem 8.8 (in the case of $\Aut(\mu)$) also appears as Teorema~1 in Tikhonov~\cite{Tikhonov2003}. His proof is somewhat different from ours and relies on King's theorem about the generic existence of square roots. A more general result was also obtained by Ageev roughly at the same time; see Subsection~\ref{ss:AutM} below.
\end{remark*}

\begin{cor}\label{c:res-cat2}
Let $G$ be either $\Aut(\mu)$ or $\Iso(\bU)$. Let $P$ be a property of abelian Polish groups such that the set
\[
\set{\pi \in \Hom(\Z, G) : \cl{\pi(\Z)} \text { has property } P}
\]
has the Baire property. Then the following are equivalent for any positive integer $d$:
\begin{enumerate} \romanenum
\item \label{i:c:res-cat2} for the generic $\pi \in \Hom(\Z, G)$, $\cl{\pi(\Z)}$ has property $P$;
\item \label{ii:c:res-cat2} for the generic $\pi \in \Hom(\Z^d, G)$, $\cl{\pi(\Z^d)}$ has property $P$.
\end{enumerate}
\end{cor}
\begin{proof}
\eqref{i:c:res-cat2} $\Rightarrow$ \eqref{ii:c:res-cat2}. Let $A$ be the set of $\pi \in \Hom(\Z, G)$ such that $\cl{\pi(\Z)}$ has $P$. Let $\Res \colon \Hom(\Z^d, G) \to \Hom(\langle e_1 \rangle, G)$ be the restriction map. We have by Theorem~\ref{t:res-cat1} that $\Res^{-1}(A)$ is comeager in $\Hom(\Z^d, G)$, and Theorem~\ref{t:generic-monothetic} shows that $B=\set{\pi \in \Hom(\Z^d, G) : \cl{\pi(\Z^d)}= \cl{\pi(\langle e_1 \rangle)}}$ is comeager. Then for any $\pi \in \Res^{-1}(A) \cap B$, we have that $\cl{\pi(\Z^d)}$ has property $P$.

\eqref{ii:c:res-cat2} $\Rightarrow$ \eqref{i:c:res-cat2}. Both $\Aut(\mu)$ and $\Iso(\bU)$ have a dense conjugacy class (for $\Aut(\mu)$, this is a classical result of Rokhlin and for $\Iso(\bU)$, it is due to Kechris--Rosendal~\cite{Kechris2007a}) and therefore if $P$ is not generic in $\Hom(\Z, G)$, its negation is, and we can apply the above argument to the negation of $P$.
\end{proof}

Foreman and Weiss~\cite{Foreman2004} have shown that conjugacy classes are meager in $\Hom(\Gamma, \Aut(\mu))$ for any amenable group $\Gamma$. We prove a result in that direction for $\Iso(\bU)$, generalizing the result of Kechris stating that conjugacy classes in $\Iso(\bU)$ are meager.

\begin{theorem}\label{t:conj-meager}
Let $\Gamma$ be a non-trivial, torsion-free, countable, abelian group. Then conjugacy classes in $\Hom(\Gamma,\Iso(\bU))$ are meager.
\end{theorem}
\begin{proof}
We use Rosendal's technique \cite{Rosendal2009a} of considering topological similarity classes. For the remainder of the proof, fix $\gamma \in \Gamma \setminus\{1\}$.

Let $S$ be an infinite subset of $\N$. We first claim that
\[
\set{\pi \in \Hom(\Gamma,G) \colon \exists n \in S \ \pi(\gamma^n)=1}
\]
is dense in $\Hom(\Gamma,\Iso(\bU))$. To see this, pick a non-empty open subset $U$. There exists a finitely generated subgroup $\Delta \cong \Z^d$ containing $\gamma$ and a non-empty open $V \sub \Hom(\Delta,\Iso(\bU))$ such that, denoting by $\Res$ the restriction map from $\Hom(\Gamma,\Iso(\bU))$ to $\Hom(\Delta, \Iso(\bU))$, one has
\[
\set{\pi \in \Hom(\Gamma,\Iso(\bU)) \colon \Res(\pi) \in V } \sub U .
\]
Let $\gamma_1,\ldots,\gamma_d$ be generators of $\Delta$.
Applying Theorem \ref{l:primeorder-metric} to $\Delta$, and using the fact that $S$ is infinite, we can find $\pi \in \Hom(\Delta,\Iso(\bU))$ and $n \in S$ such that $\pi(\gamma_i^n)=1$ for all $i \in \{1,\ldots,d\}$ (hence also $\pi(\gamma^n)=1)$. Let $\Delta'=\prod_{i=1}^d \Z(n)$; embed $\Delta'$ in $\T^d$ and let $q \colon \Delta \to \T^d$ the natural morphism, which we may extend to a morphism (still denoted by $q$) from $\Gamma$ to $\T^d$.

We can now find a finite $A \sub \bU$ and an action $q(\Delta) \actson^{\pi} A$ such that, for any action $q(\Gamma) \actson^{\phi} \bU$ extending $q(\Delta) \actson^\pi A$, one has $\phi \circ q \in U$.
Applying Lemma \ref{l:metric-ind} to $q(\Delta) \le q(\Gamma)$, using Kat\v{e}tov's construction and the homogeneity of $\bU$, we obtain such an action $\phi$, and then $\phi \circ q \in U$ and $\phi \circ q (\gamma^n)=1$. This proves the claim.

Now, let $V$ be a non-empty open subset of $\Iso(\bU)$ containing $1$ and $S$ an infinite subset of $\N$. It is clear that
\[
\set{\pi \in \Hom(\Gamma,\Iso(\bU)) \colon  \exists n \in S \ \pi(\gamma^n) \in V}
\]
is open, and the claim implies that it is dense. It follows that, for any infinite $S \sub \N$, the set
\[
A_S=\set{\pi \colon \exists (s_n) \in S^{\N} \colon \pi(\gamma^{s_n}) \to 1}
\]
is comeager. This set is conjugacy-invariant, hence if some $\pi$ had a comeager conjugacy class, $\pi$ would belong to $A_S$ for all $S$, which would imply that $\pi(\gamma^n)$ converges to $1$. This is only possible if $\pi(\gamma)=1$; since this is true for all $\gamma \in \Gamma \setminus\{1\}$, $\pi$ is trivial and we obtain a contradiction.
\end{proof}

We finally turn to studying centralizers of generic representations. Recall that $\mcC(\pi)$ denotes the centralizer of $\pi$.
\begin{theorem} \label{t:centralizers}
Let $G$ be either $\Aut(\mu)$ or $\Iso(\bU)$ and $\Gamma$ be a countable, non-trivial, torsion-free, abelian group. Then for the generic $\pi \in \Hom(\Gamma, G)$, $\mcC(\pi) = \cl{\pi(\Gamma)}$.
\end{theorem}
\begin{proof}
Denote by $\Res \colon \Hom(\Gamma \times \Z, G) \to \Hom(\Gamma, G)$ the restriction map and note that for any $\pi \in \Hom(\Gamma, G)$, we can naturally identify $\mcC(\pi)$ with $\Res^{-1}(\set{\pi})$. By Theorem~\ref{t:generic-monothetic},
\[
\forall^* \psi \in \Hom(\Gamma \times \Z, G) \ \cl{\psi(\Gamma \times \Z)} = \cl{\psi(\Gamma)}.
\]
Applying Lemma~\ref{l:commuting-categorypreserving} and Theorem~\ref{t:KU2} to $\Res$, we obtain
\[
\forall^*\pi \in \Hom(\Gamma, G) \forall^* g \in \mcC(\pi) \ g \in \cl{\pi(\Gamma)}.
\]
As $\cl{\pi(\Gamma)}$ is a closed subgroup in $\mcC(\pi)$, this implies the conclusion.
\end{proof}
\begin{remark*}
In the case $\Gamma = \Z$ and $G = \Aut(\mu)$, Theorem~\ref{t:centralizers} is due to Chacon--Sxhwartzbauer \cite{Chacon1969}. The proof presented here is different and self-contained.
\end{remark*}
\begin{remark*}
As pointed out by S. Solecki, the methods of this section, together with Theorem~\ref{t:KU2}, can be used to prove Lemma~3 in \cite{Solecki2012}. Solecki's lemma, however, had been obtained prior to our work.
\end{remark*}


\subsection{Improvements in the case of $\Aut(\mu)$} \label{ss:AutM}
Using some additional theory, it is possible to extend, in the ergodic-theoretic case, the results of the previous subsection to more general abelian groups. The following result of Ageev generalizes the theorem of King~\cite{King2000} that the restriction map $\Res \colon \Hom(\Z, \Aut(\mu)) \to \Hom(n\Z, \Aut(\mu))$ is category-preserving.
\begin{theorem}[Ageev~\cite{Ageev2003}] \label{th:Ageev}
Let $\Delta \leq \Gamma$ be countable abelian groups such that $\Delta$ is infinite cyclic. Then $\Res \colon \Hom(\Gamma, \Aut(\mu)) \to \Hom(\Delta, \Aut(\mu))$ is category-preserving.
\end{theorem}

This theorem is stated in \cite{Ageev2003} without a proof. Its formulation there is slightly different from ours but using the fact that the action of $\Aut(\mu)$ on the set of free actions of $\Gamma$ is minimal (see, e.g., \cite{Foreman2004}) and applying Proposition~\ref{p:minimal}, it is easy to derive the conclusion above.

In a subsequent work \cite{Melleray2012}, the first author has extended Theorem~\ref{th:Ageev} by proving that the restriction map from $\Hom(\Gamma,\Aut(\mu))$ to $\Hom(\Delta,\Aut(\mu))$ is category-preserving whenever $\Gamma$ is countable abelian and $\Delta \le \Gamma$ is finitely generated.

Using Theorem~\ref{th:Ageev}, one can generalize Theorem~\ref{t:centralizers} as follows.
\begin{theorem}
Let $\Gamma$ be a countable abelian group containing an infinite cyclic subgroup. Then, for a generic $\pi \in \Hom(\Gamma,\Aut(\mu))$, one has $\mcC(\pi)=\cl{\pi(\Gamma)}$.
\end{theorem}
\begin{proof}
Let $\Delta \le \Gamma$ be an infinite cyclic subgroup; by Ageev's result, the restriction map
$\pi \colon \Hom(\Gamma,\Aut(\mu)) \to \Hom(\Delta,\Aut(\mu))$ is category-preserving, so for the generic $\pi \in \Hom(\Gamma,\Aut(\mu))$, the centralizer $\mcC(\pi|_{\Delta})$ coincides with $\cl{\pi(\Delta)}$ (this follows from the fact that the centralizer of a generic element is the closure of its powers, i.e., Theorem~\ref{t:centralizers} in the case $\Delta=\Z$). Since $\mcC(\pi) \subseteq \mcC(\pi|_{\Delta})$, we obtain that $\mcC(\pi) \le \cl{\pi(\Delta)} \le \cl{\pi(\Gamma)}$ for the generic $\pi$. The converse inclusion is obvious.
\end{proof}

Using a similar argument, one also sees that for the case $G = \Aut(\mu)$, Corollary~\ref{c:res-cat2} generalizes from $\Z^d$ to any abelian group containing an element of infinite order.

We conclude with several open problems.

\smallskip \noindent \textbf{Questions:}
\begin{enumerate}
\item Are conjugacy classes in $\Hom(\Gamma, \Aut(\mu))$ and $\Hom(\Gamma, \Iso(\bU))$ meager for every countable, infinite group $\Gamma$?

\item For $G$ either $\Aut(\mu)$ or $\Iso(\bU)$, does there exist a Polish group $H$ such that for the generic $\pi \in \Hom(\Z, G)$, $\cl{\pi(\Z)} \cong H$? In particular, is this true for $H = L^0(\T)$?

\item For which pairs of groups $\Delta \le \Gamma$ is the restriction map $\Res \colon \Hom(\Gamma,G) \to \Hom(\Delta,G)$ category-preserving (again for both $G = \Aut(\mu)$ and $G = \Iso(\bU)$)?
\end{enumerate}

%

\appendix
\section{A generalization of the Kuratowski--Ulam theorem} \label{a:category}
In this appendix, we establish a generalization of the Kuratowski--Ulam theorem (Theorem~\ref{t:KU2} below), similar to the disintegration theorem in measure theory. We believe this theorem to be of some independent interest.

Let us begin by pointing out that from the proof of the classical Kuratowski--Ulam theorem (see for example, \cite{Kechris1995}*{8K}), one can obtain the following result.
\begin{theorem} \label{t:KU1}
Let $X$, $Y$ be Polish spaces and $f \colon X \to Y$ be a continuous, open map. Let $A$ be a Baire measurable subset of $X$. Then the following are equivalent:
\begin{enumerate} \romanenum
\item \label{eq:KU1(i)} $A$ is comeager in $X$;
\item \label{eq:KU1(ii)} $\forall^* y \in Y \ A \cap f^{-1}(\{y\})$ is comeager in $f^{-1}(\{y\})$.
\end{enumerate}
\end{theorem}
\begin{proof}
We begin by proving that \eqref{eq:KU1(i)} implies \eqref{eq:KU1(ii)} in the particular case when $A$ is a dense open subset of $X$. It is enough to prove that
$$\forall^* y \in Y \ A \cap f^{-1}(\{y\}) \text{ is dense in } f^{-1}(\{y\}). $$
Fix a countable basis $\set{U_n}$ for the topology of $X$; using Baire's theorem, the above statement is equivalent to saying that
$$\forall n \ \forall^* y \in Y \ \left( f^{-1}(\{y\}) \cap U_n \ne \emptyset \Rightarrow A \cap f^{-1}(\{y\}) \cap U_n \ne \emptyset \right)\ . $$
Fix $n < \omega$. The fact that $f$ is open implies that
$$C_n=\{y \in Y \colon  f^{-1}(\{y\}) \cap U_n \ne \emptyset \Rightarrow A \cap f^{-1}(\{y\}) \cap U_n \ne \emptyset \} $$
is $G_{\delta}$ (it is the union of a closed set and an open set), and to show that it is comeager, we simply need to show that it is dense in $Y$. To that end, pick a non-empty open subset $V$ of $Y$. We may restrict our attention to the case when $V \subseteq f(U_n)$. Then
$f^{-1}(V) \cap U_n \cap A$ is non-empty since $A$ is dense, and for any $x \in f^{-1}(V) \cap A \cap U_n$ we have, letting $y=f(x)$, that $y \in V$ and $f^{-1}(\{y\}) \cap A \cap U_n \ne \emptyset$. This proves that \eqref{eq:KU1(i)} implies \eqref{eq:KU1(ii)} when $A$ is a dense open subset of $X$. The general case follows easily from Baire's theorem.

Now we turn to the proof that \eqref{eq:KU1(ii)} implies \eqref{eq:KU1(i)}. We proceed by contradiction and assume that $A$ is a Baire measurable subset of $X$ that satisfies condition \eqref{eq:KU1(ii)} but is not comeager. Using the fact that $A$ is Baire measurable, we know that there exists some non-empty open subset $O$ of $X$ such that $A \cap O$ is meager, whence (using the fact that \eqref{eq:KU1(i)} implies \eqref{eq:KU1(ii)} applied to $A \cap O$)
$$
\forall^* y \in Y \ (A\cap O) \cap f^{-1}(\{y\}) \text{ is meager in }f^{-1}(\{y\}).
$$
Since $f(O)$ is non-empty and open in $Y$, we may apply the above condition and \eqref{eq:KU1(ii)} to pick $y \in f(O)$ such that $(A \cap O) \cap f^{-1}(\{y\})$ is meager in $f^{-1}(\{y\})$, and $A \cap f^{-1}(\{y\})$ is comeager in $f^{-1}(\{y\})$. This contradicts Baire's theorem in the closed non-empty set $f^{-1}(\{y\})$, and we are done.
\end{proof}

\begin{defn}
Let $X,Y$ be Polish spaces and $f \colon X \to Y$ a continuous map. We say that $f$ is \emph{category-preserving} if $f^{-1}(A)$ is comeager in $X$ whenever $A$ is comeager in $Y$.
\end{defn}
The definition above is equivalent to saying that $f^{-1}(A)$ is meager whenever $A$ is meager, and also to the fact that $f^{-1}(O)$ is dense whenever $O$ is open and dense in $Y$. Also, note that the composition of two category-preserving maps is still category-preserving.

Let us point out the following easy fact.
\begin{prop} \label{p:somewheredense}
Let $X,Y$ be Polish spaces, and $f \colon X \to Y$ a continuous map. Then the following are equivalent:
\begin{enumerate} \romanenum
\item\label{eq:somewhere(i)} $f$ is category-preserving;
\item\label{eq:somewhere(ii)} for any non-empty open subset $U$ of $X$, $f(U)$ is not meager;
\item\label{eq:somewhere(iii)} for any non-empty open subset $U$ of $X$, $f(U)$ is somewhere dense.
\end{enumerate}
\end{prop}
\begin{proof}
Suppose that \eqref{eq:somewhere(i)} holds but \eqref{eq:somewhere(ii)} does not, and let $U$ be a non-empty open subset of $X$ such that $f(U)$ is meager. Then $f^{-1}(f(U))$ contains $U$ and does not intersect the comeager set $f^{-1}(Y \setminus f(U))$, a contradiction.

It is immediate that \eqref{eq:somewhere(ii)} implies \eqref{eq:somewhere(iii)}. To see that \eqref{eq:somewhere(iii)}$\Leftrightarrow$\eqref{eq:somewhere(i)}, observe that \eqref{eq:somewhere(iii)} is a rephrasing of the condition that $f^{-1}(O)$ is dense whenever $O$ is open and dense in $Y$.
\end{proof}
In particular, any open map between Polish spaces is category-preserving.

Our aim is to obtain an extension of Theorem~\ref{t:KU1} to category-preserving maps. The following easy lemma is a crucial ingredient.
\begin{lemma}\label{l:catpreserving}
Let $X,Y$ be Polish metric spaces, and $f \colon X \to Y$ be a continuous map. Then there exists a dense $G_{\delta}$ subset $A \subseteq Y$ such that $f \colon f^{-1}(A) \to A$ is open. (We allow here $f^{-1}(A) = \emptyset$.)
\end{lemma}
\begin{proof}
Fix a countable basis $\set{U_n}_{n< \omega}$ of the topology of $X$. Since each $f(U_n)$ is analytic, there exists for all $n$ an open $O_n \sub Y$ and a meager $M_n \sub Y$ such that $f(U_n) \sdiff O_n = M_n$.
Then $B=Y \setminus (\bigcup_{n< \omega}  M_n)$ is comeager, and we can find a dense $G_{\delta}$ subset $A$ of $Y$ which is contained in $B$. We claim that this $A$ satisfies the conclusion of the lemma. To see this, note first that for any $n$, we have that $f(U_n) \cap A$ is open in $A$. From this, it follows that $f(U) \cap A$ is open in $A$ for any open subset $U$ of $X$ and we are done.
\end{proof}

We are now ready to prove the following strengthening of Theorem~\ref{t:KU1}.
\begin{theorem} \label{t:KU2}
Let $X,Y$ be Polish spaces and $f \colon X \to Y$ be a continuous, category-preserving map. Let $\Omega$ be a Baire measurable subset of $X$. Then the following are equivalent:
\begin{enumerate} \romanenum
\item \label{eq:KU2(i)} $\Omega$ is comeager in $X$;
\item \label{eq:KU2(ii)} $\forall^* y \in Y \ \, \Omega \cap f^{-1}(\{y\})$ is comeager in $f^{-1}(\{y\})$.
\end{enumerate}
\end{theorem}
\begin{proof} From Lemma~\ref{l:catpreserving} we obtain a dense $G_{\delta}$ subset $A$ of $Y$ such that $f \colon f^{-1}(A) \to A$ is open. Since $A$ is $G_{\delta}$ and $f$ is continuous, $B=f^{-1}(A)$ is $G_{\delta}$ in $X$, and it is dense since $f$ is category-preserving.

Assume that \eqref{eq:KU2(i)} holds. Since $\Omega$ is comeager in $X$ and $B$ is dense $G_{\delta}$, it follows from Baire's theorem that $\Omega \cap B$ is comeager in $B$.
Then we may apply Theorem~\ref{t:KU1} to the open map $f \colon B \to A$ to obtain that
$$\forall^* y \in Y \ ( \Omega \cap B) \cap f^{-1}(\{y\}) \text{ is comeager in } f^{-1}(\{y\}),$$
which implies the desired conclusion.

Now assume that \eqref{eq:KU2(ii)} holds. As $A$ is comeager in $Y$, we have
$$\forall^* a \in A \ \, \Omega \cap f^{-1}(\{a\}) \text{ is comeager in } f^{-1}(\{a\}).$$
Then Theorem~\ref{t:KU1} yields that $\Omega \cap B$ is comeager in $B$, and this implies that $\Omega$ is comeager in $X$.
\end{proof}

It may be worth pointing out that Theorem~\ref{t:KU2} is best possible in the sense that, if $f \colon X \to Y$ is a continuous map between two Polish spaces such that \eqref{eq:KU2(i)} and \eqref{eq:KU2(ii)} are equivalent for any Baire measurable subset $\Omega$ of $X$, then $f$ must be category-preserving.

It is particularly interesting, given a Polish group $G$ and two Polish $G$-spaces $X$ and $Y$, to obtain conditions that imply that a $G$-map $\pi \colon X \to Y$ is category-preserving.
An easy example of such a condition is the following.
\begin{prop}
Let $G$ be a Polish group and $X,Y$ be two Polish $G$-spaces. Let
\[
A= \set{y \in Y \colon G \cdot y \text{ is not meager}}
\]
and $\pi \colon X \to Y$ be a continuous $G$-map such that $\pi^{-1}(A)$ is dense in $X$. Then $\pi$ is category-preserving.
\end{prop}
\begin{proof}
Let $U$ be a non-empty open subset of $X$. By assumption, we may pick $x \in U$ such that $G \cdot \pi(x)$ is not meager. Find an open subset $V$ of $G$ such that $V \cdot x \subseteq U$. Then $\pi(V \cdot x) = V \cdot \pi(x)$ is not meager, since a countable union of translates of $V\cdot \pi(x)$ covers $G \cdot \pi(x)$. Hence $\pi(U)$ is not meager and we are done.
\end{proof}

This result applies in particular when there exists $y \in Y$ such that $G \cdot y$ is comeager and $\pi^{-1}(G \cdot y)$ is dense in $X$.

Another situation in which we obtain the same conclusion is the following.
\begin{prop} \label{p:minimal}
Let $G$ be a Polish group and $X$ and $Y$ Polish $G$-spaces such that $X$ is minimal (i.e., every orbit is dense). If $\pi \colon X \to Y$ is a $G$-map and $\pi(X)$ is comeager, then $\pi$ is category-preserving.
\end{prop}
\begin{proof}
Let $G_0$ be a dense countable subgroup of $G$ and $U \sub X$ be open non-empty. Since the action $G \actson X$ is minimal, we have $G_0 \cdot U = G \cdot U = X$ and hence $\pi(X) = \pi(G_0 \cdot U) = G_0 \cdot \pi(U)$ is comeager in $Y$, hence $\pi(U)$ is non-meager and $\pi$ is therefore category-preserving.
\end{proof}

\bibliography{mybiblio}

\def\cprime{$'$}
\begin{bibdiv}
\begin{biblist}

\bib{Ageev2000}{article}{
      author={Ageev, O.~N.},
       title={The generic automorphism of a {L}ebesgue space conjugate to a
  {$G$}-extension for any finite abelian group {$G$}},
        date={2000},
        ISSN={0869-5652},
     journal={Dokl. Akad. Nauk},
      volume={374},
      number={4},
       pages={439\ndash 442},
      review={\MR{1798480 (2001m:37004)}},
}

\bib{Ageev2003}{article}{
      author={Ageev, O.~N.},
       title={On the genericity of some nonasymptotic dynamic properties},
        date={2003},
        ISSN={0042-1316},
     journal={Uspekhi Mat. Nauk},
      volume={58},
      number={1(349)},
       pages={177\ndash 178},
         url={http://dx.doi.org/10.1070/RM2003v058n01ABEH000596},
      review={\MR{1992135 (2004c:37006)}},
}

\bib{Akcoglu1970}{article}{
      author={Akcoglu, M.~A.},
      author={Chacon, R.~V.},
      author={Schwartzbauer, T.},
       title={Commuting transformations and mixing},
        date={1970},
        ISSN={0002-9939},
     journal={Proc. Amer. Math. Soc.},
      volume={24},
       pages={637\ndash 642},
      review={\MR{0254212 (40 \#7421)}},
}

\bib{Bekka2008}{book}{
      author={Bekka, Bachir},
      author={de~la Harpe, Pierre},
      author={Valette, Alain},
       title={Kazhdan's property ({$T$})},
      series={New Mathematical Monographs},
   publisher={Cambridge University Press},
     address={Cambridge},
        date={2008},
      volume={11},
        ISBN={978-0-521-88720-5},
      review={\MR{MR2415834}},
}

\bib{Conze1972}{article}{
      author={Conze, J.~P.},
       title={Entropie d'un groupe ab\'elien de transformations},
        date={1972/73},
     journal={Z. Wahrscheinlichkeitstheorie und Verw. Gebiete},
      volume={25},
       pages={11\ndash 30},
      review={\MR{0335754 (49 \#534)}},
}

\bib{Chacon1969}{article}{
      author={Chacon, R.~V.},
      author={Schwartzbauer, T.},
       title={Commuting point transformations},
        date={1969},
     journal={Z. Wahrscheinlichkeitstheorie und Verw. Gebiete},
      volume={11},
       pages={277\ndash 287},
      review={\MR{0241600 (39 \#2939)}},
}

\bib{Folland1995}{book}{
      author={Folland, Gerald~B.},
       title={A course in abstract harmonic analysis},
      series={Studies in Advanced Mathematics},
   publisher={CRC Press},
     address={Boca Raton, FL},
        date={1995},
        ISBN={0-8493-8490-7},
      review={\MR{MR1397028 (98c:43001)}},
}

\bib{Fuchs1970}{book}{
      author={Fuchs, L{\'a}szl{\'o}},
       title={Infinite abelian groups. {V}ol. {I}},
      series={Pure and Applied Mathematics, Vol. 36},
   publisher={Academic Press},
     address={New York},
        date={1970},
      review={\MR{0255673 (41 \#333)}},
}

\bib{Foreman2004}{article}{
      author={Foreman, Matthew},
      author={Weiss, Benjamin},
       title={An anti-classification theorem for ergodic measure preserving
  transformations},
        date={2004},
        ISSN={1435-9855},
     journal={J. Eur. Math. Soc. (JEMS)},
      volume={6},
      number={3},
       pages={277\ndash 292},
      review={\MR{MR2060477 (2005c:37006)}},
}

\bib{Gao2009a}{book}{
      author={Gao, Su},
       title={Invariant descriptive set theory},
      series={Pure and Applied Mathematics (Boca Raton)},
   publisher={CRC Press},
     address={Boca Raton, FL},
        date={2009},
      volume={293},
        ISBN={978-1-58488-793-5},
      review={\MR{2455198 (2011b:03001)}},
}

\bib{Glasner1998a}{article}{
      author={Glasner, Eli},
       title={On minimal actions of {P}olish groups},
        date={1998},
        ISSN={0166-8641},
     journal={Topology Appl.},
      volume={85},
      number={1-3},
       pages={119\ndash 125},
        note={8th Prague Topological Symposium on General Topology and Its
  Relations to Modern Analysis and Algebra (1996)},
      review={\MR{MR1617456 (99c:54057)}},
}

\bib{Gruenberg1957}{article}{
      author={Gruenberg, K.~W.},
       title={Residual properties of infinite soluble groups},
        date={1957},
        ISSN={0024-6115},
     journal={Proc. London Math. Soc. (3)},
      volume={7},
       pages={29\ndash 62},
      review={\MR{0087652 (19,386a)}},
}

\bib{Gromov1983}{article}{
      author={Gromov, M.},
      author={Milman, V.~D.},
       title={A topological application of the isoperimetric inequality},
        date={1983},
        ISSN={0002-9327},
     journal={Amer. J. Math.},
      volume={105},
      number={4},
       pages={843\ndash 854},
      review={\MR{MR708367 (84k:28012)}},
}

\bib{Giordano2007}{article}{
      author={Giordano, Thierry},
      author={Pestov, Vladimir},
       title={Some extremely amenable groups related to operator algebras and
  ergodic theory},
        date={2007},
        ISSN={1474-7480},
     journal={J. Inst. Math. Jussieu},
      volume={6},
      number={2},
       pages={279\ndash 315},
         url={http://dx.doi.org/10.1017/S1474748006000090},
      review={\MR{2311665 (2008e:22007)}},
}

\bib{Glasner2005a}{article}{
      author={Glasner, E.},
      author={Weiss, B.},
       title={Spatial and non-spatial actions of {P}olish groups},
        date={2005},
        ISSN={0143-3857},
     journal={Ergodic Theory Dynam. Systems},
      volume={25},
      number={5},
       pages={1521\ndash 1538},
      review={\MR{MR2173431 (2006h:37003)}},
}

\bib{Hodges1993a}{article}{
      author={Hodges, Wilfrid},
      author={Hodkinson, Ian},
      author={Lascar, Daniel},
      author={Shelah, Saharon},
       title={The small index property for {$\omega$}-stable
  {$\omega$}-categorical structures and for the random graph},
        date={1993},
        ISSN={0024-6107},
     journal={J. London Math. Soc. (2)},
      volume={48},
      number={2},
       pages={204\ndash 218},
         url={http://dx.doi.org/10.1112/jlms/s2-48.2.204},
      review={\MR{MR1231710 (94d:03063)}},
}

\bib{Hewitt1979}{book}{
      author={Hewitt, Edwin},
      author={Ross, Kenneth~A.},
       title={Abstract harmonic analysis. {V}ol. {I}},
     edition={Second Edition},
      series={Grundlehren der Mathematischen Wissenschaften [Fundamental
  Principles of Mathematical Sciences]},
   publisher={Springer-Verlag},
     address={Berlin},
        date={1979},
      volume={115},
        ISBN={3-540-09434-2},
      review={\MR{551496 (81k:43001)}},
}

\bib{Kechris1995}{book}{
      author={Kechris, Alexander~S.},
       title={Classical descriptive set theory},
      series={Graduate Texts in Mathematics},
   publisher={Springer-Verlag},
     address={New York},
        date={1995},
      volume={156},
        ISBN={0-387-94374-9},
      review={\MR{MR1321597 (96e:03057)}},
}

\bib{Kechris2010}{book}{
      author={Kechris, Alexander~S.},
       title={Global aspects of ergodic group actions},
      series={Mathematical Surveys and Monographs},
   publisher={American Mathematical Society},
     address={Providence, RI},
        date={2010},
      volume={160},
        ISBN={978-0-8218-4894-4},
      review={\MR{2583950}},
}

\bib{King2000}{article}{
      author={King, Jonathan},
       title={The generic transformation has roots of all orders},
        date={2000},
        ISSN={0010-1354},
     journal={Colloq. Math.},
      volume={84/85},
      number={part 2},
       pages={521\ndash 547},
        note={Dedicated to the memory of Anzelm Iwanik},
      review={\MR{1784212 (2001h:37005)}},
}

\bib{Kerr2010}{article}{
      author={Kerr, David},
      author={Li, Hanfeng},
      author={Pichot, Mika{\"e}l},
       title={Turbulence, representations, and trace-preserving actions},
        date={2010},
        ISSN={0024-6115},
     journal={Proc. Lond. Math. Soc. (3)},
      volume={100},
      number={2},
       pages={459\ndash 484},
         url={http://dx.doi.org/10.1112/plms/pdp036},
      review={\MR{2595746}},
}

\bib{Kechris2007a}{article}{
      author={Kechris, Alexander~S.},
      author={Rosendal, Christian},
       title={Turbulence, amalgamation and generic automorphisms of homogeneous
  structures},
        date={2007},
     journal={Proc. Lond. Math. Soc.},
      volume={94},
      number={2},
       pages={302\ndash 350},
}

\bib{Kechris2011}{unpublished}{
      author={Kechris, Alexander~S.},
      author={Tucker-Drob, Robin~D.},
       title={The complexity of classification problems in ergodic theory},
        date={2011},
        note={preprint},
}

\bib{Lang:Algebra}{book}{
      author={Lang, Serge},
       title={Algebra},
     edition={Revised third edition},
      series={Graduate Texts in Mathematics},
   publisher={Springer-Verlag},
     address={New York},
        date={2002},
      volume={211},
        ISBN={0-387-95385-X},
      review={\MR{1878556 (2003e:00003)}},
}

\bib{TopAppl:Urysohn2008}{proceedings}{
      editor={Leiderman, Arkady},
      editor={Pestov, Vladimir},
      editor={Rubin, Matatyahu},
      editor={Sołecki, Sławomir},
      editor={Uspenskij, Vladimir~V.},
       title={{S}pecial issue: {W}orkshop on the {U}rysohn space},
      series={Topology and its Applications},
        date={2008},
      volume={155},
      number={14},
        note={Held at Ben-Gurion University of the Negev, Beer Sheva, May
  21--24, 2006},
}

\bib{Melleray2012}{unpublished}{
      author={Melleray, Julien},
       title={Extensions of generic measure-preserving actions},
        date={2012},
        note={preprint available at
  \url{http://math.univ-lyon1.fr/~melleray/GenericExtensions.pdf }},
}

\bib{Pestov2002}{article}{
      author={Pestov, Vladimir~G.},
       title={Ramsey-{M}ilman phenomenon, {U}rysohn metric spaces, and
  extremely amenable groups},
        date={2002},
        ISSN={0021-2172},
     journal={Israel J. Math.},
      volume={127},
       pages={317\ndash 357},
      review={\MR{MR1900705 (2003f:43001)}},
}

\bib{Pestov2006}{book}{
      author={Pestov, Vladimir~G.},
       title={Dynamics of infinite-dimensional groups},
      series={University Lecture Series},
   publisher={American Mathematical Society},
     address={Providence, RI},
        date={2006},
      volume={40},
        ISBN={978-0-8218-4137-2; 0-8218-4137-8},
        note={The Ramsey-Dvoretzky-Milman phenomenon, Revised edition of {\it
  Dynamics of infinite-dimensional groups and Ramsey-type phenomena} [Inst.
  Mat. Pura. Apl. (IMPA), Rio de Janeiro, 2005; MR2164572]},
      review={\MR{MR2277969}},
}

\bib{Pestov2007a}{incollection}{
      author={Pestov, Vladimir~G.},
       title={Forty-plus open problems on large topological groups},
        date={2007},
   booktitle={Open problems in topology. {II}},
      editor={Pearl, Elliott},
   publisher={Elsevier B. V., Amsterdam},
       pages={439\ndash 450},
      review={\MR{2367385 (2008j:54001)}},
}

\bib{Prasad1981}{article}{
      author={Prasad, V.~S.},
       title={Generating dense subgroups of measure preserving
  transformations},
        date={1981},
        ISSN={0002-9939},
     journal={Proc. Amer. Math. Soc.},
      volume={83},
      number={2},
       pages={286\ndash 288},
      review={\MR{MR624915 (83b:28020)}},
}

\bib{Pestov2006a}{article}{
      author={Pestov, Vladimir~G.},
      author={Uspenskij, Vladimir~V.},
       title={Representations of residually finite groups by isometries of the
  {U}rysohn space},
        date={2006},
        ISSN={0970-1249},
     journal={J. Ramanujan Math. Soc.},
      volume={21},
      number={2},
       pages={189\ndash 203},
      review={\MR{2244544 (2007f:43005)}},
}

\bib{Rosendal2009a}{article}{
      author={Rosendal, Christian},
       title={The generic isometry and measure preserving homeomorphism are
  conjugate to their powers},
        date={2009},
        ISSN={0016-2736},
     journal={Fund. Math.},
      volume={205},
      number={1},
       pages={1\ndash 27},
         url={http://dx.doi.org/10.4064/fm205-1-1},
      review={\MR{2534176 (2010g:03071)}},
}

\bib{Rosendal2010p2}{unpublished}{
      author={Rosendal, Christian},
       title={Finitely approximable groups and actions part {I}{I}: Generic
  representations},
        date={2010},
        note={Journal of Symbolic Logic, to appear},
}

\bib{Stepin2004}{article}{
      author={Stepin, A.~M.},
      author={Eremenko, A.~M.},
       title={Nonuniqueness of an inclusion in a flow and the vastness of a
  centralizer for a generic measure-preserving transformation},
        date={2004},
        ISSN={0368-8666},
     journal={Mat. Sb.},
      volume={195},
      number={12},
       pages={95\ndash 108},
         url={http://dx.doi.org/10.1070/SM2004v195n12ABEH000867},
      review={\MR{2138483 (2006b:37008)}},
}

\bib{Slutsky2011}{unpublished}{
      author={Slutsky, Konstantin},
       title={Non-genericity phenomenon in some ordered {F}ra\"iss\'e classes},
        date={2011},
        note={preprint},
}

\bib{Solecki2012}{unpublished}{
      author={Solecki, S{\l}awomir},
       title={Closed subgroups generated by generic measure automorphisms},
        date={2012},
        note={preprint available at
  \url{http://www.math.uiuc.edu/~ssolecki/papers/gen_meas_aut_fin2.pdf}},
}

\bib{Schreier1933}{article}{
      author={Schreier, J.},
      author={Ulam, S.},
       title={\"{U}ber die {P}ermutationsgruppe der nat\"urlichen
  {Z}achlenfolge},
        date={1933},
     journal={Studia Math.},
      volume={4},
       pages={134\ndash 141},
}

\bib{Tikhonov2003}{article}{
      author={Tikhonov, S.~V.},
       title={The generic action of the group {$\Bbb Z^d$} can be embedded into
  the action of the group {$\Bbb R^d$}},
        date={2003},
        ISSN={0869-5652},
     journal={Dokl. Akad. Nauk},
      volume={391},
      number={1},
       pages={26\ndash 28},
      review={\MR{2041884 (2004k:37007)}},
}

\end{biblist}
\end{bibdiv}
\end{document}